\documentclass{article}
\usepackage[utf8]{inputenc}
\usepackage[a4paper, total={6in, 8in}]{geometry}
\usepackage{authblk}
\usepackage{amsthm}
\usepackage{amsmath}
\usepackage{amssymb}
\usepackage{graphicx}
\usepackage{bbm}
\usepackage{braket,amsfonts,amssymb}
\usepackage[linesnumbered,ruled]{algorithm2e}
\usepackage{comment}
\usepackage{amsopn}
\usepackage{bm}
\usepackage{mathrsfs}  
\usepackage[mathscr]{euscript}
\usepackage{bm,nicefrac}
\usepackage{extarrows}
\usepackage{caption}
\usepackage{subcaption}
\usepackage{xcolor}
\usepackage[linktoc=all]{hyperref}
\usepackage{titlesec}
\usepackage{authblk}
\hypersetup{
    colorlinks,
    citecolor=blue,
    filecolor=black,
    linkcolor=blue,
    urlcolor=black
}

\def\a{a}
\def\b{b}
\def\bo{{\bf 1}}
\def\cc{Q}
\def\c{c}

\def\e{{\bf e}}

\def\g{g}

\def\n{n}
\def\k{k}

\def\s{s}
\def\t{t}
\def\u{{\bf u}}
\def\v{{\bf v}}
\def\x{{\bf x}}

\def\y{{\bf y}}
\def\z{{\bf z}}
\def\ue{u}
\def\uprod{\tilde{\u}}
\def\uprode{\tilde{\ue}}
\def\uprodinfty{\uprod_\infty}

\def\ve{v}
\def\vprod{\tilde{\v}}
\def\vprode{\tilde{\ve}}
\def\vprodinfty{\tilde{\v}_\infty}

\def\w{{\bf w}}
\def\we{w}
\def\wprod{\tilde{\w}}
\def\wprode{\tilde{\we}}
\def\wprodinfty{\tilde{\w}_\infty}
\def\xe{x}
\def\xprod{\tilde{\bf x}}
\def\xprode{\tilde{\xe}}
\def\xkc{\tilde{\bf x}_{k^c}}
\def\xprodinfty{\xprod_\infty}

\def\ye{y}
\def\ze{z}

\def\A{{\bf A}}
\def\Ae{A}

\def\F{F}

\def\L{L}
\def\M{M}
\def\N{N}
\def\Q{Q}
\def\R{\mathbb{R}}
\def\T{\text{T}}
\def\Spm{S_{\pm}}

\def\0{\boldsymbol{0}}
\def\1{\boldsymbol{1}}

\def\Lquad{\mathcal{L}_{\text{quad}}}
\def\Lover{\mathcal{L}_{\text{over}}}

\def\IzPlus{I_{\z,+}}
\def\IzMinus{I_{\z,-}}

\DeclareMathOperator*{\argmin}{arg\,min}

\renewcommand{\L}{L}

\renewcommand{\epsilon}{\varepsilon}

\newcommand{\rev}[1]{{\color{black} #1}}

\providecommand{\keywords}[1]{\textbf{\textit{Keywords --- }} #1}

\newtheorem{theorem}{Theorem}[section]

\newtheorem{definition}[theorem]{Definition}
\newtheorem{lemma}[theorem]{Lemma}
\newtheorem{corollary}[theorem]{Corollary}
\theoremstyle{remark}
\newtheorem{remark}[theorem]{Remark}

\title{More is Less: Inducing Sparsity via Overparameterization}
\author[1]{Hung-Hsu Chou}
\author[1]{Johannes Maly}
\author[2]{Holger Rauhut}
\affil[1]{Ludwig Maximilian University of Munich, Germany}
\affil[2]{RWTH Aachen University, Germany}

\date{}

\begin{document}

\maketitle



\begin{abstract}
In deep learning it is common to overparameterize neural networks, that is, to use more parameters than training samples. Quite surprisingly training the neural network via (stochastic) gradient descent leads to models that generalize very well, while classical statistics would suggest overfitting. In order to gain understanding of this implicit bias phenomenon we study the special case of sparse recovery (compressed sensing) which is of interest on its own. More precisely, in order to reconstruct a vector from underdetermined linear measurements, we introduce a corresponding overparameterized square loss functional, where the vector to be reconstructed is deeply factorized into several vectors. We show that, if there exists an exact solution, vanilla gradient flow for the overparameterized loss functional converges to a good approximation of the solution of minimal $\ell_1$-norm. The latter is well-known to promote sparse solutions. As a by-product, our results significantly improve the sample complexity for compressed sensing via gradient flow/descent on overparameterized models derived in previous works. The theory accurately predicts the recovery rate in numerical experiments. Our proof relies on analyzing a certain Bregman divergence of the flow. This bypasses the obstacles caused by non-convexity and should be of independent interest.

\end{abstract}



\tableofcontents


\section{Introduction}
\label{sec:Introduction}
Overparameterization is highly successful in learning with deep neural networks. While this empirical finding was likely observed by countless practitioners, it was systematically studied in numerical experiments in \cite{neyshabur2017geometry,neyshabur2015,zhang17}. Increasing the number of parameters far beyond the number of training samples leads to better generalization properties of the learned networks. This is in stark contrast to classical statistics, which would rather suggest overfitting in this scenario. 
The loss function typically has infinitely many global minimizers in this setting (there are usually infinitely many networks fitting the training samples exactly in the overparameterized regime \cite{zhang17}), so that the employed optimization algorithm has a significant influence on the computed solution. The commonly used (stochastic) gradient descent and its variants seem to have an implicit bias towards ``nice'' networks with good generalization properties. It is conjectured that in fact, (stochastic) gradient descent applied to learning deep networks favors solutions of low complexity.
Of course, the right notion of ``low complexity'' needs to be identified and may depend on the precise scenario, i.e., network architecture. In the simplified setting of linear networks, i.e., matrix factorizations, and the problem of matrix recovery or more specifically matrix completion, several works
\cite{arora2018optimization,arora2019implicit,chou2020implicit,geyer2019implicit,Bach2019implicit,Gissin2019Implicit,gunasekar2019implicit,gunasekar2017implicit,neyshabur2017geometry,neyshabur2015,razin2020implicit,Razin2021implicitTensor,Razin2022hierachicalTensor,soudry2018implicit,stoger2021small} identified the right notion of low complexity to be low rank of the factorized matrix (or tensor, in some these works). Nevertheless, despite initial theoretical results, a fully convincing theory is not yet available even for this simplified setting.

As a contribution to this line of research, this article studies the performance of overparameterized models in the context of the classical compressed sensing problem, where the goal is to recover an unknown high-dimensional signal $\x^\star \in \R^\N$ from few linear measurements of the form
\begin{align} \label{eq:CS}
    \y = \A \x^\star \in \R^\M,
\end{align}
where $\A \in \R^{\M\times \N}$ models the measurement process. Whereas this is in general not possible for $\M < \N$, the seminal works \cite{candes2006robust,candes2006stable,donoho2006compressed} showed that unique reconstruction of $\x^\star$ from $\A$ and $\y$ becomes feasible if $\x^\star$ is $\s$-sparse, $\A$ satisfies certain conditions (e.g. restricted isometry property), and $\M$ scales like $\s\log(\N/\s)$. While it is well-known that
$\x^\star$ can be recovered via $\ell_1$-minimization, i.e.,
\begin{align}
    \label{eq:BP}
    \x^\star = \argmin_{\A\z = \y} \| \z \|_1,
\end{align}
also known as basis pursuit, surprisingly enough, we observe similar recovery with gradient descent on overparameterized models, which originate from a very different background. (Here in the following, $\|\cdot\|_p$ denotes the standard $\ell_p$-norm for $1 \leq p \leq \infty$, and $\x_0>0$ means all entries of $\x_0$ are strictly positive.) Our goal is to understand this connection via gradient flow -- a continuous version of gradient descent -- in this context.

In order to gain a first understanding of the power of overparameterization, we compare the quadratic loss (without overparameterization)
\begin{align}\label{eq:Loss_quad}
    \Lquad(\x) := \frac{1}{2}\|\A \x-\y\|_2^2
\end{align}
and its overparameterized versions
\begin{align} \label{eq:L_over}
    \Lover\big(\x^{(1)}, \dots, \x^{(\L)}\big)
    & := \frac{1}{2}\Big\|\A\big(\x^{(1)}\odot \cdots\odot \x^{(\L)}\big)-\y\Big\|_2^2, \\
    \Lover^{\pm} \big(\u^{(1)}, \dots, \u^{(\L)}, \v^{(1)}, \dots, \v^{(\L)}\big)
    & := \frac{1}{2}\Big\|\A\Big( \bigodot_{\k=1}^\L \u^{(\k)} - \bigodot_{\k=1}^\L \v^{(\k)}  \Big)-\y\Big\|_2^2,\label{eq:L_over_refined}
\end{align}
where $\odot$ is the Hadamard (or entry-wise) product. It turns out that the results from minimizing different loss functions via gradient flow/descent are quite different. Being initialized with $\x(0) = \0$, on the quadratic loss $\Lquad$ gradient flow and descent converge to the least-squares solution
\begin{align*}
    \x_\infty
    := \lim_{\t\to\infty} \x(\t)
    = \argmin_{\A \z = \y} \|\z\|_2\;,
\end{align*}
which is unrelated to the sparse ground-truth $\x^\star$ in general. On the other hand, gradient flow and descent applied to $\Lover$ and $\Lover^{\pm}$ often lead to sparse results, despite the absence of explicit regularization. This can be viewed as a particular instance of implicit bias/regularization. In this work, we restrict ourselves to gradient flow under identical initialization --- all factors $\x^{(k)}$ resp. $\u^{(k)}$ and $\v^{(k)}$ are initialized with the the same vector --- and formally define the gradient flow for $\Lover$ as
\begin{align*}
    \rev{\frac{d}{d\t}}\x^{(\k)}(t) = -\nabla_{\x^{(\k)}}\Lover(\x^{(1)}(t),\hdots,\x^{(L)}(t)), \qquad \x^{(\k)}(0) = \x_0>0, \quad \k = 1,\hdots,L.
\end{align*}
(The gradient flows $\u^{(\k)}(t)$ and $\v^{(\k)}(t)$ for $\Lover^\pm$ are defined similarly.) Due to the identical initialization, the factors $\x^{(k)}$ resp. $\u^{(k)}$ and $\v^{(k)}$ stay identical over time, which relates to the so-called balancedness condition \cite{arora2019implicit,Razin2021implicitTensor}. This allows us to simplify the functionals $\Lover$ and $\Lover^\pm$ to the \emph{reduced factorized loss} functions
\begin{align}
    \mathcal{L}(\x) &:= \frac{1}{2}\|\A\x^{\odot L} -\y\|_2^2 \label{eq:L} \\
    \mathcal{L}^\pm(\u,\v) &:= \frac{1}{2}\|\A(\u^{\odot \L} - \v^{\odot L}) - \y\|_2^2 \label{eq:L_refined},
\end{align}
where $\z^{\odot L} = \z \odot \cdots \odot \z$ denotes the $L$-th Hadamard power of a vector. Let us mention that this common simplification already appeared in \cite{Gissin2019Implicit,woodworth2020kernel}. For the sake of completeness, we provide the derivation in Appendix~\ref{sec:ModelReduction}. 

Our results show that, under the above initialization, the product $\xprod(t):=\bigodot_{\k=1}^\L \x^{(\k)}$ (and $\bigodot_{\k=1}^\L \u^{(\k)} - \bigodot_{\k=1}^\L \v^{(\k)}$) converges to an approximate solution of the $\ell_1$-minimization problem \eqref{eq:CS}, provided that the initialization parameter $\alpha > 0$ is sufficiently small. Hence, in situations where $\ell_1$-minimization is known to successfully recover sparse solutions also overparameterized gradient flow will succeed. In particular, conditions on the restricted isometry property or the null space property on $\A$, 
tangent cone conditions and dual certificates on $\A$ and $\x^*$ ensuring recovery of $s$-sparse vectors 
transfer to overparameterized gradient flow. For instance, a random Gaussian matrix $\A \in \R^{M \times N}$ ensures recovery of $s$ sparse vectors via overparameterized gradient flow for $M \sim s \log(N/s)$. We refer the interested reader to the monograph \cite{foucart2013compressed} for details on compressed sensing.

Our main result improves on previous work: For $\L=2$, it was shown in \cite{li2021implicit} that gradient 
descent for $\mathcal{L}^\pm$ defined in \eqref{eq:L_refined} converges to an $s$-sparse $\x^*$ if $\A$ satisfies a certain coherence assumption, which requires at least $M \gtrsim s^2$ measurements. In \cite{vaskevicius2019implicit}, it was shown for general $L \geq 2$ that 
$\mathcal{L}^\pm$ converges to an $s$-sparse $\x^*$ if the restricted isometry constant $\delta_s$ of $\A$ (see below) essentially satisfies 
$\delta_s \leq c/\sqrt{s}$. This condition can only be satisfied if $M \gtrsim s^2 \log(N/s)$, see \cite{foucart2013compressed}. Hence, our result significantly reduces the required number of measurements, in fact, down to the optimal number. However, we note that
\cite{li2021implicit,vaskevicius2019implicit} work with gradient descent, while we use gradient flow. Extending our result to gradient descent is left to future work.

\subsection{Main results}

As discussed above, we consider vanilla gradient flow on the reduced factorized models \eqref{eq:L} and \eqref{eq:L_refined}. We show that if the solution space is non-empty, gradient flow converges to a solution of \eqref{eq:CS} whose $\ell_1$-norm is close to the minimum among all solutions. (Here $\x_0>0$ means all entries of $\x_0$ are strictly positive.) We now state our main result.

\begin{theorem}
\label{theorem:L1_equivalence}
    Let $\L\geq 2$, $\A\in\mathbb{R}^{\M\times\N}$, and $\y\in\mathbb{R}^{\M}$. For the reduced loss function $\mathcal L^{\pm}$ defined in \eqref{eq:L_refined}
    let $\u(\t)$ and $\v(\t)$ follow the flow
    \begin{align*}
        &\u'(\t) = -\nabla_{\u} \mathcal L^{\pm}\big(\u, \v \big),\qquad
        \u(0)=\u_0>0,\\
        &\v'(\t) = -\nabla_{\v} \mathcal L^{\pm}\big(\u, \v \big),\qquad
        \v(0)=\v_0>0.
    \end{align*}
    Let $\uprod=\u^{\odot\L}$, $\vprod=\v^{\odot\L}$. Suppose the solution set $S = \{\z \in \R^\N :\A\z = \y\}$ is non-empty. Then  
    \begin{equation}\label{xprod-lim}
        \xprodinfty:= \lim_{\t\to\infty}
    \left(\uprod(t) - \vprod(t) \right)
    \end{equation}
    exists and is contained in $S$.
    
    Moreover, denote the {\bf signed weighted} $\ell_1$-norm of $\z$ with weights $(\w_+,\w_-)$ by
    \begin{equation*}
        \|\z\|_{(\w_+,\w_-),1} = \|\w_+\odot\z_+ + \w_-\odot\z_-\|_1
    \end{equation*}
    where $\z_+,\z_-$ are the positive and the negative part of $\z$, respectively. Let $\gamma = 1-\frac{2}{\L}\in[0,1)$. Consider
    \begin{align*}
        \w_+ = \uprod(0)^{\odot-\gamma}
        &,\quad\w_- = \vprod(0)^{\odot-\gamma},\\
        \beta_{1} = \|\uprod(0)\|_{\w_+,1} + \|\vprod(0)\|_{\w_-,1}
        &,\quad\beta_{\min}=\min_{\n\in[\N]}\we_\n\min(\uprode_\n(0),\vprode_\n(0)).
    \end{align*}
    Let $\Q := \min_{\z\in S} \|\z\|_{(\w_+,\w_-),1}$. Suppose $\Q>\beta_1$. Then $\|\xprodinfty\|_{(\w_+,\w_-),1}-\Q \leq \epsilon\Q$, where $\epsilon$ is defined as\rev{
    \begin{align}\label{eq:L1min_general}
        \epsilon:=\begin{cases}
            \frac{\log(\beta_{1}/\beta_{\min})}{\log(\Q/\beta_{1})} &\text{if }\L=2,\\[6pt]
            \frac{\L(\beta_{1}^{\gamma} - \beta_{\min}^{\gamma})}{ 2(\Q^{\gamma}-\beta_{1}^{\gamma})}&\text{if }\L>2.
        \end{cases}
    \end{align}
    }    
\end{theorem}
The proof of Theorem \ref{theorem:L1_equivalence} is given in Section \ref{sec:Gradient_Flow_General}.
\begin{remark}
One may interpret Theorem \ref{theorem:L1_equivalence} from the following perspectives.
\begin{itemize}
    \item \textbf{Effect of depth.}
    Consider the case $\L>2$. According to \eqref{eq:L1min_general}, when $\beta_{\min}\ll\beta_{1}\ll\Q$ are fixed, the error bound $\epsilon$ scales like $\beta_{1}^{1-\frac{2}{\L}}$. This indicates that the shrinkage effect of $\epsilon$ with respect to $\L$ increases as $\L$ increases.
    \item \textbf{Effect of weight.}
    For $\L=2$, we simply get $\ell_1$-minimization, while for $\L>2$ we get weighted $\ell_1$-minimization, where the weights purely depend on the initialization. More importantly, we can put different weights for the positive and negative component with respect to the same entry. If $\uprod(0) = \vprod(0)$, then we have regular weighted minimization. If furthermore $\uprod(0)=\alpha \bo$ for some $\alpha$, then we have un-weighted minimization.
    \item \textbf{Effect of $\beta_{\min}$.}
    Note that when $\beta_{\min}\ll\beta_{1}$ (which holds for large $\N$), the effect of $\beta_{\min}$ in \eqref{eq:L1min_general} becomes insignificant for $\L>2$. Hence without much loss we can upper bound $\epsilon$ by
    \begin{equation*}
        \frac{\L}{2}\cdot\frac{\beta_{1}^{1-\frac{2}{\L}}}{\Q^{1-\frac{2}{\L}}-\beta_{1}^{1-\frac{2}{\L}}},
    \end{equation*}
    which no longer depends on $\beta_{\min}$. On the other hand, the effect of $\beta_{\min}$ remains significant for $\L=2$ and cannot be removed easily. This indicates that when there is randomness in initialization, the error for $\L>2$ might be easier to bound than for the case where $\L=2$.
    \item \textbf{Generality.}
    As previously discussed in Section \ref{sec:Introduction}, Theorem \ref{theorem:L1_equivalence} applies to the general loss function $\Lover^\pm$ in \eqref{eq:L_over_refined} as well as long as all factors $\u^{(k)}$ and $\v^{(k)}$ are initialized identically with $\u_0$ and $\v_0$.
    \item \textbf{Scale invariance.}
    Equation \eqref{eq:L1min_general} provides an explicit non-asymptotic bound that is scale invariant in the following sense. Let $\lambda>0$ be a scaling factor. If we replace $\u_0$, $\v_0$, and $\y$ with $\lambda\u_0$, $\lambda\v_0$, and $\lambda^\L\y$, then $\beta_{1}$, $\beta_\infty$, and $Q$ will each be scaled by $\lambda^\L$ and the right-hand side of \eqref{eq:L1min_general} remains unaffected under the scaling.
    
    \item \textbf{Convergence rate.}
    While decreasing the initialization scale leads to decreasing approximation error, empirically it makes the gradient descent converge slower. Hence there is a trade-off between accuracy and computational complexity. 
\end{itemize}
\end{remark}

\subsection{Application to compressed sensing}

As a consequence of Theorem~\ref{theorem:L1_equivalence}, accurately reconstructing (approximately) sparse vectors from incomplete linear measurements (compressed sensing) can provably be achieved via gradient flow on $\Lover$, using the minimal amount of measurements. This suggests that gradient descent on overparametrized least-square functionals can be viewed as an alternative algorithm for compressed sensing. To be concrete let us provide example results in this direction. A matrix $\A \in \R^{\M \times \N}$ is said to satisfy the stable null space property of order $\s$ and constant $\rho \in (0,1)$ if for all vector $\v \in \ker (\A) \setminus \{0\}$ and all index sets $S \subset \{1,\hdots,\N\}$ of cardinality at most $\s$ it holds
\[
\|\v_S\|_1 \leq \rho \|\v_{S^c}\|_1,
\]
where $\v_S$ denotes the restriction of $\v$ to the entries in $S$ and $S^c$ is the complement of $S$. It is well-known that various types of random matrices satisfy the stable null space property
in an (almost) optimal parameter regime with high probability,
see for instance \cite{brugiapaglia2021sparse,foucart2013compressed,mendelson2018improved} and references therein. For instance, for a Gaussian random matrix $\A$ this is true provided that 
\begin{equation}\label{M-bound-Gaussian}
\M \gtrsim \rho^{-2} \s \log(e\N/\s).
\end{equation}
Moreover, let us also introduce the error of best $s$-term approximation in $\ell_1$ of a vector $x$ as
\[
\sigma_s(\x)_1 =  \min_{\z : \|\z\|_0 \leq \s} \|\x - \z\|_1,
\]
where $\|\z\|_0 = \#\{\ell : z_\ell \neq 0\}$ denotes the sparsity of $\z$. Clearly, $ \sigma_\s(\x)_1 = 0$ if $\x$ is $\s$-sparse.

The following result is a straightforward implication of  \cite[Theorem 4.14]{foucart2013compressed} together with Theorem \ref{theorem:L1_equivalence}.
We also refer to Figure \ref{fig:ContributionSection} and Section \ref{sec:Numerical_Experiment}
for further illustration.
\begin{corollary} \label{cor:CS}
    Let $\x_* \in \mathbb R^{\N}$ and $\y = \A\x_*$, where $\A \in \mathbb R^{\M \times \N}$ satisfies the stable null space property of order $s$ with constant $\rho\in(0,1)$. Let $\xprodinfty$ be the limit vector of gradient flow as in \eqref{xprod-lim} with initialization $\u_0=\v_0>0$. Then
    \begin{align*}
        \|\xprodinfty - \x_*\|_1
        \leq  \frac{1+\rho}{1-\rho}\left(\epsilon\|\x_*\|_1 +  2\sigma_{s}(\x_*)_1 \right),
    \end{align*}
    where $\epsilon$ is defined in \eqref{eq:L1min_general}.
\end{corollary}

We can also show stability with respect to noise on the measurements assuming a certain quotient property of the measurement matrix \cite{foucart2013compressed} in addition to the null space property. To be precise, we say that $\A \in \R^{\M \times \N}$ satisfies the $\ell_1$-quotient property with constant $d$ relative to the $\ell_2$-norm if for all $\e \in \R^\M$, there exists $\u \in \mathbb{C}^\N$ with $\A \u = \e$
such that 
\[
\| \u\|_1 \leq d \sqrt{\s_*} \|\e\|_2 \quad \mbox{ with } s_* = \M/(\log(e\N/\M).
\]
Gaussian random matrices satisfy this property for an absolute constant $d$ with high probability, see \cite[Chapter 11]{foucart2013compressed} for details. 

Moreover, we require the following strengthened version of the null space property. A matrix $\A \in \mathbb{R}^{\M \times \N}$ satisfies the $\ell_2$-robust null space property with constants $0 < \rho < 1$ and $\tau > 0$ of order $s$ (with respect to $\ell_2$) if, for any subset $S \subset [N]$ of cardinality $s$,
\begin{equation}\label{robust-nsp}
\|\v_S\|_2 \leq \frac{\rho}{\sqrt{s}} \|\v_{S^c}\|_1 + \tau \|\A \v\|_2 \quad \mbox{ for all } \v \in \R^\N. 
\end{equation}
Again, Gaussian random matrices satisfy this property with appropriate absolute constants $\tau$ and $\rho$ with high probability under \eqref{M-bound-Gaussian}.

The following theorem, which combines Theorem \ref{theorem:L1_equivalence} with \cite[Theorem 11.12]{foucart2013compressed}, establishes robustness under noise for the reconstruction of $\x_*$ via gradient flow on the overparameterized functional $\Lover^\pm$ in \eqref{eq:L_over_refined}.

\begin{theorem}\label{thm:noisy-cs} 
Let $c > 1$ be fixed. Let $\A \in \R^{\M \times \N}$ be a matrix satisfying the $\ell_2$-robust null space property with constants $0 \leq \rho < 1$ and $\tau > 0$ of order $s=cs_* = c \M/\log(e\N/\M)$, and the $\ell_1$-quotient property with respect to the $\ell_2$-norm with constant $d > 0$. For $\x_* \in \R^N$ let $\y = \A \x_*+\e$ for some noise vector $\e \in \R^M$. Let $\xprodinfty$ be the limit \eqref{xprod-lim} of the gradient flow for $\Lover^\pm$ with initialization $\u_0=\v_0>0$. Then
\begin{align*}
    \|\xprodinfty - \x_*\|_2 \leq \frac{C}{\sqrt{\s}}(\epsilon\|\x_*\|_1 + 2 \sigma_s(\x_*)_1 ) + C'\|\e\|_2.
    \|\xprodinfty - \x_*\|_1.
\end{align*}
As before, $\epsilon$ is defined in \eqref{eq:L1min_general}. The constants $C,C'>0$ only depend on $\rho, \tau, c, d$.
\end{theorem}
For $\M \times \N$ Gaussian random matrices $\A$ the assumptions of the theorem are satisfied with high probability if $\M \geq C \s \log(e \N/s)$ for an appropriate constant $C$ (depending on the other constants $\rho,\tau,d,c$).
The proof of this theorem is contained in Appendix \ref{subsec:Noise robust compressed sensing}.
\begin{figure}[t]
    \centering
    \begin{subfigure}[b]{0.32\textwidth}
        \centering
        \includegraphics[width = \textwidth]{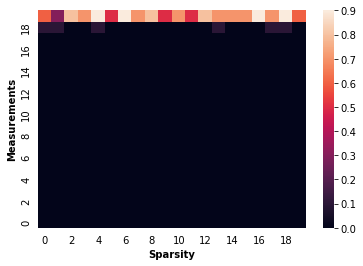}
        \subcaption{$\mathcal{L}_{\text{quad}}$ minimization \eqref{eq:Loss_quad} via GD}
    \end{subfigure}
    \hfill
    \begin{subfigure}[b]{0.32\textwidth}
        \centering
        \includegraphics[width = \textwidth]{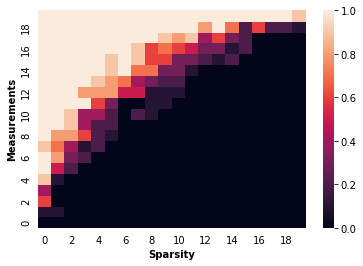}
        \subcaption{$\ell_1$ minimization \eqref{eq:BP}}
    \end{subfigure}
    \hfill
    \begin{subfigure}[b]{0.32\textwidth}
        \centering
        \includegraphics[width = \textwidth]{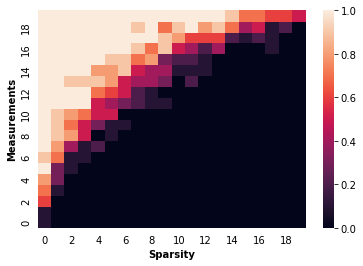}
        \subcaption{$\Lover^{\pm}$ minimization \eqref{eq:L_over_refined} via GD}
    \end{subfigure}
    \caption{We compare the recovery probability for different method via heatmaps. The horizontal axis is the sparsity level $s$, and vertical axis is the number of measurements $\M$. The result (a) shows that we cannot expect any recovery by using the naive quadratic loss. Note that our model (c) achieves similar, and in fact arguably better performance than the well known (b) basis pursuit method.}
    \label{fig:ContributionSection}
\end{figure}
%

\subsection{Related work}

Since training deep neural (linear) networks via gradient descent is connected to matrix factorization, see, e.g., \cite{Bach2019implicit,chou2020implicit,gunasekar2019implicit}, many works focus on understanding implicit bias in this setting. The corresponding results on matrix factorization and matrix sensing are of a similar flavor and show an implicit low-rank bias of gradient flow/descent when minimizing factorized quadratic losses \cite{arora2018optimization,arora2019implicit,chou2020implicit,geyer2019implicit,Bach2019implicit,Gissin2019Implicit,gunasekar2019implicit,gunasekar2017implicit,neyshabur2017geometry,neyshabur2015,razin2020implicit,soudry2018implicit,stoger2021small,wang2022large,wu2021implicit}. It is noteworthy that the existing matrix sensing results, e.g., \cite{stoger2021small}, require $\mathcal{O}(r^2 n)$ measurements to guarantee reconstruction of rank-$r$ $n\times n$-matrices via gradient descent, i.e., they share the sub-optimal sample complexity of \cite{li2021implicit,vaskevicius2019implicit}. For comparison, for low-rank matrix reconstruction by conventional methods like nuclear-norm minimization, only $\mathcal{O}(r n)$ measurements are needed. For a more detailed discussion of the literature on matrix factorization/sensing via overparametrization, we refer the reader to \cite{chou2020implicit}.

Most of the above mentioned works consider the case of $\L=2$ layers, whereas the literature for general $\L>2$ is scarce. In terms of proof method,
\cite{wu2021implicit} is most related to ours among the above works. In the setting of matrix sensing, the authors show that mirror flow/descent exhibits an implicit regularization when the mirror map is the spectral (hyper)entropy. Their analysis heavily relies on the concept of Bregman divergence. Furthermore, they draw a connection between gradient descent on symmetric matrix factorization ($\L = 2$) and the mirror descent without factorization, cf.\ \cite{azulay2021implicit,gunasekar2017implicit,gunasekar2021mirrorless}. Although having started from a different perspective, we realized that the quantity that we called \textbf{\textit{solution entropy}} in the first version of our paper can be viewed as a Bregman divergence. With this observation, we were able to further improve our work and remove technical assumptions on $\A$. Despite certain similarities in the proof strategy, the authors of \cite{wu2021implicit}, who also authored the closely related works \cite{wu2020continuous,wu2021hadamard}, implicitly concentrate on the case $\L=2$ by only considering two specific Bregman divergences (one for the symmetric and one for the asymmetric case, which need to be distinguished in the matrix setting). In contrast, our work provides suitable Bregman divergences for all $\L \ge 2$ and as such allows to analyze gradient descent on overparamterized loss functions with arbitrary depth. Let us emphasize that the deeper case $\L\geq 3$ has not been covered by any existing work to the best of our knowledge.

Due to the strong coupling of weights in factorized matrix sensing,
several existing works restrict themselves to the vector case. In \cite{li2021implicit,vaskevicius2019implicit} the authors derive robust reconstruction guarantees for gradient descent and the compressed sensing model \eqref{eq:CS}. Whereas in \cite{vaskevicius2019implicit} the authors consider the case where $\L=2$ and the sensing matrix $\A$ satisfies the restricted isometry property, \cite{li2021implicit} extends the results to $\L\ge 2$ under a coherence assumption for $\A$. In \cite{hoff2017lasso} the author shows that solving the LASSO is equivalent to minimizing an $\ell_2$-regularized overparameterized functional (for $\L = 2$) and uses this to solve LASSO by alternating least-squares methods. Although the author also considers deeper factorization, the presented approach leads to different results since the overparameterized functional is equivalent to $\ell_{\frac{2}{\L}}$-norm instead of $\ell_1$-norm minimization, for $\L > 2$. The subsequent work \cite{zhao2019implicit} builds upon those ideas to perform sparse recovery with gradient descent by assuming a restricted isometry property of $\A$. Nevertheless, the presented results share the sub-optimal sample complexity discussed above.

Instead of specific initialization, in \cite{woodworth2020kernel} the authors examine the limits of gradient flow when initialized by $\alpha \w_0$, for any $\w_0$. They show that for large $\alpha>0$ the limit of gradient flow approximates the least-square solution, whereas for $\alpha>0$ small it approximates an $\ell_1$-norm minimizer. While the authors discuss more general types of initialization, their proof strategy is fundamentally different from ours and has certain shortcomings. They need to \textit{assume} convergence of the gradient flow and obtain only for $\L=2$ non-asymptotic bounds on the initialization magnitude required for implicit $\ell_1$-regularization. In contrast, we actually \textit{show} convergence of gradient flow and provide non-asymptotic bounds for all $\L$ leading to less restrictive assumptions on the initialization magnitude. 
%

\subsection{Outline}

%
Sections \ref{sec:Gradient_Flow_Positive} and \ref{sec:Gradient_Flow_General} are dedicated to proving Theorem \ref{theorem:L1_equivalence}. Section \ref{sec:Gradient_Flow_Positive} illustrates the proof strategy in a simplified setting of positive solutions, whereas Section \ref{sec:Gradient_Flow_General} extends the proof to full generality. Finally, we present in Section~\ref{sec:Numerical_Experiment} numerical evidence supporting our claims and conclude with a brief summary/outlook on future research directions in Section \ref{sec:Summary_Future_Direction}.
%

\subsection{Notation}

%
For $\N\in\mathbb{N}$, we denote $[\N]=\{1,2,\dots,N\}$. Boldface lower-case letters like $\x$ represent vectors with entries $\xe_n$, while boldface upper-case letters like $\A$ represent matrices with entries $\Ae_{mn}$. For $\x,\y\in\mathbb{R}^\N$, $\x\geq \y$ means that $\xe_n\geq \ye_n$, for all $n \in [\N]$. We use $\odot$ to denote the Hadamard product, i.e., the vectors $\x \odot \y$ and $\x^{\odot p}$ have entries $(\x\odot \y)_n = \xe_n \ye_n$ and $(\x^{\odot p})_n = \xe_n^p$, respectively. We abbreviate $\xprod := \bigodot_{\k\in[\L]} \x^{(\k)} = \x^{(1)}\odot \cdots\odot \x^{(L)}$. The logarithm is applied entry-wise to positive vectors, i.e., $\log(\x) \in \R^\N$ with $\log(\x)_n = \log(\xe_n)$. For convenience we denote $\mathbb{R}_+^{\N} = \{\x\in\mathbb{R}^\N: \xe_n\geq 0 \;\, \forall n\in[\N]\}$.


\section{Positive Case}
\label{sec:Gradient_Flow_Positive}

To show Theorem \ref{theorem:L1_equivalence}, we are going to prove in this section the following simplified version, Theorem~\ref{theorem:L1_equivalence_positive}, that treats the model in \eqref{eq:L} and is restricted to the positive orthant (recall that gradient flow applied to \eqref{eq:L} preserves the entry-wise sign of the iterates).
Since analyzing $\mathcal L$ instead of $\mathcal L^\pm$ comes with less notational load, this approach facilitates to digest the core argument. The proof of Theorem~\ref{theorem:L1_equivalence} is then a straight-forward adaption that only requires technical fine tuning and will be discussed in Section~\ref{sec:Gradient_Flow_General}. Let us define the set of non-negative solutions
\begin{align} \label{def:Splus}
    S_+ = \{\z\geq 0:\A\z = \y\}
\end{align}
and note that Theorem \ref{theorem:L1_equivalence_positive} can be easily adapted to other orthants by changing the signs of corresponding entries of the initialization vector.

\begin{theorem}[Equivalence to $\ell_1$-minimization, positive case]\label{theorem:L1_equivalence_positive}
    Let $\L\geq 2$, $\A\in\mathbb{R}^{\M\times\N}$ and $\y\in\mathbb{R}^{\M}$ and assume that $S_+$ defined in \eqref{def:Splus} is non-empty. For the reduced loss function $\mathcal L$ in \eqref{eq:L}, define $\x(\t)$ via
    \begin{equation}\label{eq:dynamic_positive}
    \x'(t) = -\nabla_{\x} \mathcal L(\x(t)), \qquad \x(0) = \x_0>0.
    \end{equation}
    Let $\xprod=\x^{\odot L}$.Then the limit $\xprodinfty:= \lim_{\t\to\infty}
    \xprod(t)$
    exists and $\xprodinfty \in S_+$.
    
    Moreover, denote the weighted $\ell_1$-norm of $\z$ with weight $\w$ by $\|\z\|_{\w,1} := \|\w\odot\z\|_1$ and consider 
    \begin{align*}
        \w = \xprod(0)^{\odot\frac{2}{\L}-1},\quad
        \Q := \min_{\z\in S_+}\|\z\|_{\w,1},\quad
        \beta_{1} = \|\xprod(0)\|_{\w,1},\quad
        \beta_{\min} = \min_{\n\in[\N]}\we_\n\xprode_\n(0).
    \end{align*}
    \rev{Suppose $\Q>\beta_1$}, then $\|\xprodinfty\|_{\w,1}-\Q\leq\epsilon\Q$, where $\epsilon$ is defined the same way as \eqref{eq:L1min_general} in Theorem \ref{theorem:L1_equivalence}.
\end{theorem}
We will now start to introduce lemmas and theorems to prove Theorem \ref{theorem:L1_equivalence_positive}.


\subsection{Bregman divergence}
\label{subsec:Bregman_Divergence_positive}


As a crucial ingredient for Theorem \ref{theorem:L1_equivalence_positive}, we show convergence of $\x(\t)^{\odot \L}$ and characterize its limit. In order to state the corresponding result, we introduce the function
\begin{equation}\label{eq:g}
    g_{\x}(\z) = \begin{cases}
    \langle \z, \log(\z) -\bo  - \log(\x) \rangle  & \text{if }\L = 2,\\
    \langle \z, \x^{\odot \frac{2}{\L}-1} \rangle-\frac{\L}{2}\|\z\|_{2/\L}^{2/\L}  & \text{if }\L > 2.
    \end{cases}
\end{equation}

\begin{theorem}\label{theorem:Bregman_positive}
For $\L\geq 2$, let $\xprod(\t) = \x(\t)^{\odot\L}$ and
\begin{align} \label{eq:dynamicsBregman_positive}
    \x'(\t)
    =-\nabla \mathcal L (\x(\t))
    =-\L \left[\A^{\T}( \A\x^{\odot \L}(\t) -\y)\right] \odot \x^{\odot \L-1}(\t)
\end{align}
with $\x(0)\geq 0$. Assume that $S_+$ in \eqref{def:Splus} is non-empty. Then $\xprodinfty:= \lim_{\t\to\infty} \xprod(\t)$ exists and
\begin{align}\label{eq:optimal_x_positive}
    \xprodinfty
    = \argmin_{\z\in S_+} g_{\xprod(0)}(\z).
\end{align}
\end{theorem}
The proof of this theorem is based on the Bregman divergence defined as follows.
\begin{definition}[Bregman Divergence]\label{def:Bregman_Divergence}
Let $\F:\Omega\to\mathbb{R}$ be a continuously-differentiable, strictly convex function defined on a closed convex set $\Omega$. The Bregman divergence associated with $\F$ for points $p,q\in\Omega$ is defined as
\begin{equation}
    D_{\F}(p,q) = \F(p) - \F(q) - \langle \nabla \F(q), p-q \rangle.
\end{equation}
\end{definition}
By strict convexity of $\F$ it is straight-forward to verify the following.
\begin{lemma}[\cite{Bregman1967}]
   The Bregman divergence $D_\F$ is non-negative and, for any $q \in \Omega$, the function $p \mapsto D_F(p,q)$ is strictly convex.
\end{lemma}

To prove Theorem \ref{theorem:Bregman_positive}, we use the Bregman divergence with respect to $\F: \mathbb{R}_{+}^{\N} \to \mathbb{R}$ with
\begin{equation}\label{eq:Breman_positive}
    \F(\x) 
    = \left. \begin{cases}
    \frac{1}{2} \langle \x \odot \log(\x) - \x, \1 \rangle & \text{if }\L = 2\\
    \frac{\L}{2(2-\L)} \langle \x^{\odot \frac{2}{\L}}, \1 \rangle & \text{if }\L > 2
    \end{cases} \right\}
    = \begin{cases}
    \frac{1}{2}\sum_{n=1}^{\N} \xe_n\log(\xe_n) - \xe_n  & \text{if }\L = 2\\
    \frac{\L}{2(2-\L)}\sum_{n=1}^{\N} \xe_n^{\frac{2}{\L}} & \text{if }\L > 2,
    \end{cases}
\end{equation}
with the understanding that $z\log(z) = 0$ for $z=0$.
Note that $\F$ is strictly convex because its Hessian $\mathbf H_\F(\x)$ is diagonal with positive diagonal entries $\frac{1}{\L}\xe_n^{-2+\frac{2}{\L}}$, for $\x$ in the interior of $\R_+^N$, i.e., if $x_n >0$ for all $n$. Thus the Bregman divergence $D_{\F}$ is well-defined. It has the following property.

\begin{lemma} \label{lem:Boundedness}
   Let $\F$ be the function defined in \eqref{eq:Breman_positive} and $\xprod(\t) \colon \R_{+} \to \R_{+}^\N$ be a continuous function with $\xprod(0) > \0$. Let $\z \ge \0$ be fixed. If $D_\F(\z,\xprod(\t))$ is bounded, then $\| \xprod(\t) \|_2$ is bounded.
\end{lemma}

\begin{proof}
    We will prove the statement by contraposition. Let $\| \xprod(\t) \|_2$ be unbounded. Then there exists a sequence $0 < \t_1 \le \t_2 \le \dots$ such that $\| \xprod(\t_k) \|_2 \to \infty$. Hence there exists some $n \in [\N]$ and a subsequence $0 < \t_{n_1} \le \t_{n_2} \le \dots$ such that $\xprode_n(\t_{n_k}) \to \infty$.
    Note that
    \begin{align*}
        D_\F(\z,\xprod(\t))
        &= \begin{cases}
        \frac{1}{2} \big( \langle \xprod(\t) - \z \odot\log(\xprod(\t)), \1 \rangle + \langle \z\odot\log(\z) - \z, \1 \rangle \big) & \mbox{if } \L = 2,\\
        \frac{1}{2(\L-2)} \big( \langle (\L-2) \xprod(\t)^{\odot \frac{2}{\L}} + 2\z \odot \xprod(\t)^{\odot (\frac{2}{\L} - 1)}, \1 \rangle - \L \langle \z^{\odot \frac{2}{\L}}, \1 \rangle \big) & \mbox{if } \L > 2.
        \end{cases}
    \end{align*}
    For $z \geq 0$, both the function $x \mapsto x - z \log(x)$
     (for $\L=2$) and the function $x \mapsto (\L-2) x^{\frac{2}{\L}} + 2z x^{ \frac{2}{\L} - 1}$ (for $\L > 2$) are bounded from below and diverge to infinity as $x$ goes to infinity. Thus $D_F(\z,\xprod(\t_{n_k})) \to \infty$ and consequently $D_F(\z,\xprod(\t))$ is unbounded.
\end{proof}

Furthermore, we will use the following observation.

\begin{lemma}\label{lemma:Bregman_nonincreasing}
   Suppose $\x$ follows the dynamics \eqref{eq:dynamic_positive} in Theorem \ref{theorem:L1_equivalence_positive}. Let $\xprod = \x^{\odot\L}$. Then for any $\z\in S_+$,
   \begin{equation}
       \partial_\t D_{\F}(\z,\xprod(\t))
       =-2\L\cdot\mathcal{L}(\xprod(\t)^{\odot\frac{1}{\L}}).\label{eq:DFderivative}
   \end{equation}
\end{lemma}
\begin{proof}
    Note that due to continuity, for all $\t\geq 0$, $\x(\t)$ remains non-negative (if any entry $x(\t)_i$ vanishes at $\t_* > 0$, then $x(\t)_i = 0$ for all $\t \ge \t_*$ by the shape of \eqref{eq:dynamicsBregman_positive}). Suppose $\z\in S_+$. Then, we have
\begin{align*}
    \partial_\t D_{\F}(\z,\xprod(\t))
    &= \partial_\t \left[ \F(\z) - \F(\xprod(\t)) - \langle \nabla \F(\xprod(\t)), \z-\xprod(\t) \rangle \right]\\
    &= 0 - \langle \nabla \F(\xprod(\t)), \xprod'(\t) \rangle 
    - \langle \partial_\t \nabla \F(\xprod(\t)), \z-\xprod(\t) \rangle
    + \langle \nabla \F(\xprod(\t)), \xprod'(\t) \rangle\\
    &= - \langle \partial_\t \nabla \F(\xprod(\t)), \z-\xprod(\t) \rangle.
\end{align*}
By the chain rule and the diagonal shape of $\mathbf H_\F$, we know that
\begin{align*}
    \partial_t \nabla \F(\xprod(\t))
    &= \mathbf H_\F (\xprod(\t)) \cdot \xprod'(\t)
    = \frac{1}{L} \xprod(\t)^{\odot(-2+\frac{2}{L})} \odot \xprod'(\t)
    = \frac{1}{L} \x(\t)^{\odot(-2L+2)} \odot \big( L\x(\t)^{\odot(L-1)} \odot \x'(\t) \big) \\
    &= \x(\t)^{\odot(-2L+2)} \odot  \x(\t)^{\odot(L-1)} \odot \big( - \L\left[\A^{\T}( \A\x^{\odot \L}(\t) -\y)\right] \odot \x^{\odot \L-1}(\t) \big) \\
    &= -\L\left[\A^{\T}( \A\xprod(\t) -\y)\right].
\end{align*}
Therefore,
\begin{align*}
    \partial_\t D_{\F}(\z,\xprod(\t))
    &= \L\langle \A^{\T}( \A\xprod(\t) -\y), \z-\xprod(\t) \rangle\\
    &= -\L \langle \A\xprod(\t) -\y, \A\xprod(\t) - \A\z \rangle\\
    &= - \L\|\A\xprod(\t) -\y\|_2^2
     = -2\L\cdot\mathcal{L}(\xprod(\t)^{\odot\frac{1}{\L}}).
\end{align*}
This completes the proof.
\end{proof}

\begin{proof}[Proof of Theorem \ref{theorem:Bregman_positive}]
Let us begin with a brief outline of the proof. We will first use Lemma \ref{lemma:Bregman_nonincreasing} to show that the loss function converges to zero. Then we employ Lemma \ref{lem:Boundedness} to deduce the convergence of $\xprod(\t)$. Finally, we conclude with the unique characterization of the limit $\xprodinfty$ induced by the Bregman divergence. 

We first show that $\lim_{\t \to \infty} \mathcal{L}(\xprod(\t)^{\odot\frac{1}{\L}}) = 0$. Since $D_\F(\z,\xprod(\t)) \geq 0$ and $\partial_\t D_{\F}(\z,\xprod(\t))\leq 0$ by Lemma \ref{lemma:Bregman_nonincreasing}, $D_\F(\z,\xprod(\t))$ must converge for any fixed $\z\in S_+$. Since $\mathcal{L}(\xprod(\t)^{\odot\frac{1}{\L}})$ is non-negative and non-increasing in $\t$, for any $\t>0$ we have $\min_{\tau\leq \t}\mathcal{L}(\xprod(\tau)^{\odot\frac{1}{\L}}) = \mathcal{L}(\xprod(\t)^{\odot\frac{1}{\L}})$. Thus by Lemma \ref{lemma:Bregman_nonincreasing}
\begin{align*}
    D_\F(\z,\xprod(\t)) - D_\F(\z,\xprod(0)) 
    = \int_0^\t \partial_\t D_{\F}(\z,\xprod(\tau)) d\tau
    = - 2\L \int_0^\t \mathcal{L}(\xprod(\tau)^{\odot\frac{1}{\L}}) d\tau
    \leq - 2 \L \t \mathcal{L}(\xprod(\t)^{\odot\frac{1}{\L}}).
\end{align*}
Rearranging terms we have
\begin{align*}
    \mathcal{L}(\xprod(\t)^{\odot\frac{1}{\L}}) \leq \frac{D_\F(\z,\xprod(0)) -D_\F(\z,\xprod(\t))}{2\L\t} \le \frac{D_\F(\z,\xprod(0))}{2\L\t} \longrightarrow 0
\end{align*}
as $\t\to\infty$. We conclude that $\lim_{\t \to \infty} \mathcal{L}(\xprod(\t)^{\odot\frac{1}{\L}}) = 0$ and thus $\lim_{\t \to \infty} \A\xprod(\t) = \y$. 

We now deduce that $\xprod(\t)$ converges to some $\bar{z}\in S_{+}$. By Lemma \ref{lemma:Bregman_nonincreasing}, $D_\F(\z,\xprod(\t))$ is bounded. By Lemma \ref{lem:Boundedness}, $\|\xprod(\t)\|_2$ is bounded. According to the Bolzano-Weierstrass Theorem, there exists a sequence $\{\t_k\}_{k\in\mathbb{N}}$ such that $\xprod(\t_k)$ converges to some $\bar{\z}\in S_{+}$ since the loss converges to zero. Consequently, $D_\F(\bar{\z},\xprod(\t_k))$ converges to zero. Because $D_\F(\bar{\z},\xprod(\t))$ is non-increasing, this implies that $D_\F(\bar{\z},\xprod(\t))$ converges to zero as well. By strict convexity of $D_\F$ in its first component and the fact that $D_\F(\w,\w) = 0$, $\xprod$ must converge to $\bar{z}\in S_{+}$.

Because $\partial_\t D_{\F}(\z,\xprod(\t))$ is identical for all $\z \in S_+$, the difference
\begin{align}
    \Delta_{\z} = D_{\F}(\z,\xprod(0)) - D_{\F}(\z,\xprodinfty)
\end{align}
is constant in $\z \in S_+$. By non-negativity of $D_{\F}$,
\begin{align}
    D_{\F}(\z,\xprod(0))
    \geq \Delta_{\z}
    = \Delta_{\xprodinfty}
    = D_{\F}(\xprodinfty,\xprod(0)).
\end{align}
Thus
\begin{align*}
    \xprodinfty \in
    &\argmin_{\z\in S_+}D_{\F}(\z,\xprod(0))
    = \argmin_{\z\in S_+} \F(\z) - \F(\xprod(0)) - \langle \nabla \F(\xprod(0)), \z-\xprod(0) \rangle\\
    &= \argmin_{\z\in S_+} \F(\z) - \langle \nabla \F(\xprod(0)), \z\rangle\\
    &= \argmin_{\z\in S_+} \begin{cases}
    \sum_{n=1}^{\N} \ze_n\log(\ze_n) - \ze_n - \log(\xprode_n(0))\ze_n  & \text{if }\L = 2,\\
    \sum_{n=1}^{\N} - \ze_n^{\frac{2}{\L}} + \frac{2}{\L}\xprode_n(0)^{\frac{2}{\L}-1}\ze_n & \text{if }\L > 2,
    \end{cases}\\
    &= \argmin_{\z\in S_+} \begin{cases}
    \langle \z, \log(\z) -\bo - \log(\xprod(0)) \rangle  & \text{if }\L = 2,\\
    \langle \z, \xprod(0)^{\odot \frac{2}{\L}-1} \rangle-\frac{\L}{2}\|\z\|_{2/\L}^{2/\L}  & \text{if }\L > 2
    \end{cases}\\
    &= \argmin_{\z\in S_+} g_{\xprod(0)}(\z).
\end{align*}
This completes the proof.
\end{proof}
According to Theorem \ref{theorem:Bregman_positive},
\begin{equation}\label{eq:g_compare_1}
    g_{\xprod(0)}(\xprodinfty) \leq g_{\xprod(0)}(\z)
\end{equation}
for all $\z\in S_+ \subset \mathbb{R}_+^\N$. The proof of Theorem \ref{theorem:L1_equivalence_positive} now uses that $g_{\xprod(0)}(\z)\approx\|\z\|_{\w,1}$ when $\xprod(0)$ is small such that \eqref{eq:g_compare_1} implies that $\|\x\|_{\w,1}\leq \|\z\|_{\w,1}+\epsilon$, for some small $\epsilon$. Before turning to the proof of Theorem \ref{theorem:L1_equivalence_positive}, we thus investigate the relation between $g_{\x}(\z)$ and $\|\z\|_{\w,1}$ in the following lemma.
\begin{lemma}\label{lemma:KKT_g_z}
   Consider the function $\g_\x:\mathbb{R}_+^\N\to \mathbb{R}$, which was defined in \eqref{eq:g} and appeared in Theorem \ref{theorem:Bregman_positive}, given by
   \begin{equation*}
        g_{\x}(\z) := \begin{cases}
        \langle \z, \log(\z) -\bo  - \log(\x) \rangle  & \text{if }\L = 2,\\
        \langle \z, \x^{\odot \frac{2}{\L}-1} \rangle-\frac{\L}{2}\|\z\|_{2/\L}^{2/\L}  & \text{if }\L > 2.
        \end{cases}
    \end{equation*}
    Suppose $\x>0$ and $\z\geq0$. Denote the {\bf unsigned weighted} $\ell_1$-norm of $\z$ with weight $\w$ by
    \begin{equation*}
        \|\z\|_{\w,1} = \|\w\odot\z\|_1.
    \end{equation*}
    Let $\w = \x^{\odot\frac{2}{\L}-1}$, $\beta_{1} = \|\x\|_{\w,1}$, and $\beta_{\min} = \min_{\n\in[\N]}\we_\n\xe_\n$. Then \begin{equation}
        \tilde{\g}(\|\z\|_{\w,1},\beta_{1}) \leq \g_{\x}(\z)\leq \tilde{\g}(\|\z\|_{\w,1},\beta_{\min}).
    \end{equation}
    where $\tilde \g:\mathbb{R}_+\times\mathbb{R}_+\to\mathbb{R}$ is defined as
    \begin{equation}
        \tilde{\g}(\a,\b)=\begin{cases}
        \a(\log\a-1-\log\b)  & \text{if }\L = 2,\\
        \a - \frac{\L}{2}\a^{\frac{2}{\L}}\b^{1-\frac{2}{\L}} & \text{if }\L > 2.
        \end{cases}
    \end{equation}
\end{lemma}
\begin{proof}
    Note that the theorem holds trivially for $\|\z\|_{\w,1} =0$. Hence we only consider the case $\|\z\|_{\w,1}>0$. Since $\z\in\{\boldsymbol\xi\geq 0:\|\boldsymbol\xi\|_{\w,1} = \|\z\|_{\w,1}\}$,
    \begin{align}
        g_{\x}(\z) &\leq \sup_{\boldsymbol\xi\geq 0,\,\|\boldsymbol\xi\|_{\w,1} = \|\z\|_{\w,1}} g_{\x}(\boldsymbol\xi)\label{eq:g_sup},\\
        g_{\x}(\z) &\geq \inf_{\boldsymbol\xi\geq 0,\,\|\boldsymbol\xi\|_{\w,1} = \|\z\|_{\w,1}} g_{\x}(\boldsymbol\xi)\label{eq:g_inf}.
    \end{align}
    Note that on such domain, $g_{\x}(\boldsymbol\xi)$ can be expressed as
    \begin{align*}
        g_{\x}(\boldsymbol\xi)
        = \begin{cases}
        \langle \boldsymbol\xi, \log(\boldsymbol\xi) - \log(\x) \rangle - \|\z\|_{\w,1}  & \text{if }\L = 2,\\
        \|\z\|_{\w,1} -\frac{\L}{2}\|\boldsymbol\xi\|_{2/\L}^{2/\L}  & \text{if }\L > 2.
        \end{cases}
    \end{align*}
    Hence while taking the supremum/infimum, we only need to concentrates on the term $ \langle \boldsymbol\xi, \log(\boldsymbol\xi) - \log(\x) \rangle$ for $\L=2$, and the term $-\frac{\L}{2}\|\boldsymbol\xi\|_{2/\L}^{2/\L}$ for $\L>2$. Since those terms are convex (in $\boldsymbol\xi$) and the domain is a convex set, the supremum is attained at an extreme point of the form $\boldsymbol\xi=\c{\bf e}_\n$ for some $\c>0$ where ${\bf e}_\n$ is a coordinate vector. To determine $\c$, we use the relation $\|\boldsymbol\xi\|_{\w,1} = \|\z\|_{\w,1}$ to deduce that
    \begin{equation*}
        \c = \frac{\|\z\|_{\w,1}}{\we_\n} = \|\z\|_{\w,1}\xe_\n^{1-\frac{2}{\L}}.
    \end{equation*}
    Hence
    \begin{align*}
        \sup_{\boldsymbol\xi\geq 0,\,\|\boldsymbol\xi\|_{\w,1} = \|\z\|_{\w,1}} g_{\x}(\xi)
        = \sup_{\n\in[\N]}g_{\x}(\c{\bf e}_\n)
        &= \sup_{\n\in[\N]}\begin{cases}
        \c\log(\c) - \c\log(\xe_\n) - \|\z\|_{\w,1}  & \text{if }\L = 2,\\
        \|\z\|_{\w,1} -\frac{\L}{2}\c^{\frac{2}{\L}}  & \text{if }\L > 2.
        \end{cases}\\
        &= \sup_{\n\in[\N]} \begin{cases}
        \|\z\|_{\w,1} (\log\|\z\|_{\w,1} -1 - \log(\xe_\n))  & \text{if }\L = 2,\\
        \|\z\|_{\w,1} -\frac{\L}{2}\|\z\|_{\w,1}^{\frac{2}{\L}}\xe_\n^{(1-\frac{2}{\L})\frac{2}{\L}}  & \text{if }\L > 2,
        \end{cases}\\
        &= \begin{cases}
        \|\z\|_{\w,1} (\log\|\z\|_{\w,1} -1 - \log(\min_{\n\in[\N]}\xe_\n))  & \text{if }\L = 2,\\
        \|\z\|_{\w,1} -\frac{\L}{2}\|\z\|_{\w,1}^{\frac{2}{\L}}\min_{\n\in[\N]}\xe_\n^{(1-\frac{2}{\L})\frac{2}{\L}}  & \text{if }\L > 2,
        \end{cases}\\
        &= \tilde{\g}(\|\z\|_{\w,1},\beta_{\min}).
    \end{align*}
    Let us now consider the minimization problem \eqref{eq:g_inf}. 
    The minimum is either attained in the interior of $\R_+^N$ or on the boundary, where at least one coordinate of $\boldsymbol\xi$ is zero. In order to treat both cases simultaneously we fix a set $I \subset [\N]$ of non-zero coordinates of $\boldsymbol\xi$ (possibly $I = [\N]$), set $\xi_n = 0$ for $n \in I$ 
    and restrict the minimization to the positive coordinates, i.e., to $\boldsymbol\xi^I \in \R^I$ with $\xi_n > 0$ for all $n \in I$. With $\x^I \in \R^I$ defined via $\xe^I_\n = \xe_\n$ for $n \in I$, we set
    \[
    g_{I,\x^I}(\boldsymbol\xi^I):= \begin{cases}
        \langle \boldsymbol\xi^I, \log(\boldsymbol\xi^I) - \log(\x^I) \rangle  & \text{if }\L = 2,\\
        -\frac{\L}{2}\|\boldsymbol\xi^I\|_{2/\L}^{2/\L}  & \text{if }\L > 2.
        \end{cases} 
    \]
    If $\xi_n = 0$ for $n \notin I$ 
    and $\xi^I_n = \xi_n$ for $n \in I^c$, then we have $g_{I,\x^I}(\boldsymbol\xi^I) = g_\x(\boldsymbol\xi)$, so that we consider now the minimization problem
    \begin{equation}\label{min-prob-I}
        \min_{\boldsymbol\xi^I \in \R^I, \xi^I > 0} g_{I,\x^I}(\boldsymbol\xi^I) \quad \mbox{ subject to }
        \begin{cases}
        \langle \boldsymbol\xi^I, \bo^I\rangle = \|\z\|_{\w,1} &\text{if }\L=2,\\
        \langle \boldsymbol\xi^I, (\x^I)^{\odot \frac{2}{\L}-1}\rangle = \|\z\|_{\w,1}&\text{if }\L>2.
        \end{cases}
    \end{equation}
    In order to solve this optimization problem, we consider the Lagrangian
    \begin{equation}
        h_I(\boldsymbol\xi^I,\lambda) = g_{I,\x^I}(\boldsymbol\xi^I) + \lambda\cdot
        \begin{cases}
        \langle \boldsymbol\xi^I, \bo^I\rangle - \|\z\|_{\w,1} &\text{if }\L=2,\\
        \langle \boldsymbol\xi^I, (\x^I)^{\odot \frac{2}{\L}-1}\rangle - \|\z\|_{\w,1}&\text{if }\L>2.
        \end{cases}
    \end{equation}
    Its gradient is given, for $\boldsymbol\xi^I > 0$, by
    \begin{align*}
        \nabla_{\boldsymbol\xi^I} h_I(\boldsymbol\xi^I,\lambda) &= 
        \begin{cases}
        (\lambda+1)\cdot\bo^I +\log(\boldsymbol\xi^I) -\log(\x^I)  & \text{if }\L = 2,\\
        \lambda(\x^I)^{\odot \frac{2}{\L}-1} -(\boldsymbol\xi^I)^{\odot \frac{2}{\L}-1} & \text{if }\L > 2,\\
        \end{cases}\\
        \nabla_\lambda h(\boldsymbol\xi^I,\lambda) &=
        \begin{cases}
        \langle \boldsymbol\xi^I, \bo^I\rangle - \|\z\|_{\w,1} &\text{if }\L=2,\\
        \langle \boldsymbol\xi^I, (\x^I)^{\odot \frac{2}{\L}-1}\rangle - \|\z\|_{\w,1} &\text{if }\L>2.
        \end{cases}
    \end{align*}
    At a minimum $\boldsymbol\xi^I_*$ for the problem \eqref{min-prob-I} we necessarily have $\nabla h_I(\boldsymbol\xi^I_*, \lambda_*) = 0$ for some $\lambda_* \in \R$. This leads to  
    \begin{align*}
        \begin{cases}
        (\xi^I_*)_\n = e^{-\lambda_*-1 }\xe^I_\n  &\text{if }\L = 2,\\
        (\xi^I_*)_n^{\frac{2}{\L}-1} = \lambda_* (\xe_\n^I)^{\frac{2}{\L}-1} &\text{if }\L > 2,
        \end{cases}
    \end{align*}
    for all $\n\in I$. By assumption all components of $\x$ are positive so that $(\xi^I_*)_\n > 0$ for all $\n \in I$ can be satisfied. Moreover, for a $(\xi^I_*)_\n > 0$ the relation above for $L>2$ implies that $\lambda_* > 0$.
    Taking into account the constraint that $\|\boldsymbol\xi^I\|_{\w,1}= \|\z\|_{\w,1}$, we obtain $e^{-\lambda_*-1} = \frac{\|\z\|_{\w,1}}{\|\x^I\|_{\w,1}}$ for $L=2$. Similarly, for $\L > 2$ we have $\lambda_* = \frac{\|\z\|_{\w,1}}{\|\x^I\|_{\w,1}}$.
    Combining the results we have
    \begin{equation*}
        (\xi^I_*)_\n=\xe_\n \frac{\|\z\|_{\w,1}}{\|\x^I\|_{\w,1}} \quad \mbox{ for all } \n \in I.
    \end{equation*}
    For notation simplicity, denote $\c = \frac{\|\z\|_{\w,1}}{\|\x^I\|_{\w,1}}$. This implies that the quantity in \eqref{eq:g_inf} can be expressed as
    \begin{align*}
        \inf_{\substack{\boldsymbol\xi\geq 0\\ \|\boldsymbol\xi\|_{\w,1} = \|\z\|_{\w,1}}} g_{\x}(\boldsymbol\xi)
        = \inf_{I\subset[\N]}g_{\x}\left(\c\x^I\right)
        &= \inf_{I\subset[\N]}\begin{cases}
        \langle \c\x^I, \log(\c\x^I)  - \log(\x^I) \rangle - \|\z\|_{\w,1}  & \text{if }\L = 2,\\
        \|\z\|_{\w,1}-\frac{\L}{2}\|\c\x^I\|_{2/\L}^{2/\L}  & \text{if }\L > 2.
        \end{cases}\\
        &= \inf_{I\subset[\N]}\begin{cases}
        \langle \c\log(\c)\|\x^I\|_1-\|\z\|_{\w,1} & \text{if }\L = 2,\\
        \|\z\|_{\w,1} - \frac{\L}{2}\c^{\frac{2}{\L}}\|\x^I\|_{\w,1}  & \text{if }\L > 2.
        \end{cases}\\
        &= \inf_{I\subset[\N]}\begin{cases}
        \|\z\|_{\w,1} (\log\|\z\|_{\w,1} -1 - \log\|\x^I\|_{\w,1})  & \text{if }\L = 2,\\
        \|\z\|_{\w,1} - \frac{\L}{2}\|\z\|_{\w,1}^\frac{2}{\L}\|\x^I\|_{\w,1}^{1-\frac{2}{\L}} & \text{if }\L > 2,
        \end{cases}\\
        &= \begin{cases}
        \|\z\|_{\w,1} (\log\|\z\|_{\w,1} -1 - \log\|\x\|_{\w,1})  & \text{if }\L = 2,\\
        \|\z\|_{\w,1} - \frac{\L}{2}\|\z\|_{\w,1}^\frac{2}{\L}\|\x\|_{\w,1}^{1-\frac{2}{\L}} & \text{if }\L > 2,
        \end{cases}\\
        &= \tilde{\g}(\|\z\|_{\w,1},\beta_{1}).
    \end{align*}
This completes the proof.
\end{proof}
\begin{proof}[Proof of Theorem \ref{theorem:L1_equivalence_positive}]
Let $\z\in S_+$, We will first discuss the case $\L=2$. By Theorem \ref{theorem:Bregman_positive}, $g_{\xprod(0)}(\xprodinfty) \leq g_{\xprod(0)}(\z)$. According to Lemma \ref{lemma:KKT_g_z}, this implies that 
\begin{equation}\label{eq:xz_compare_1}
    \tilde{\g}(\|\xprodinfty\|_{\w,1},\beta_{1})
    \leq \g_{\xprod(0)}(\xprodinfty)
    \leq \g_{\xprod(0)}(\z)
    \leq \tilde{\g}(\|\z\|_{\w,1},\beta_{\min}).
\end{equation}
It follows from \eqref{eq:xz_compare_1} that
\begin{equation*}
    \tilde{\g}(\|\xprodinfty\|_{\w,1},\beta_{1}) - \tilde{\g}(\|\z\|_{\w,1},\beta_{1})
    \leq \tilde{\g}(\|\z\|_{\w,1},\beta_{\min}) - \tilde{\g}(\|\z\|_{\w,1},\beta_{1}).
\end{equation*}
Since $\tilde{\g}$ is convex in its first argument, we have
\begin{align}
    \partial_1\tilde{\g}(\|\z\|_{\w,1},\beta_{1})\cdot (\|\xprodinfty\|_{\w,1}-\|\z\|_{\w,1})
    &\leq\tilde{\g}(\|\xprodinfty\|_{\w,1},\beta_{1}) - \tilde{\g}(\|\z\|_{\w,1},\beta_{1})\nonumber\\
    &\leq\tilde{\g}(\|\z\|_{\w,1},\beta_{\min}) - \tilde{\g}(\|\z\|_{\w,1},\beta_{1}),
    \label{eq:convex_approx}
\end{align}
where $\partial_1$ denotes the derivative with respect to the first argument.
Note that $\Q=\min_{\z\in S_+}\|\z\|_{\w,1}$ satisfies $\Q>\beta_{1}$ by assumption. Computing the derivative then yields that $\partial_1\tilde{\g}(\Q,\beta_{1})>0$. Dividing both sides of \eqref{eq:convex_approx} by $\partial_1\tilde{\g}(\Q,\beta_{1})$, we thus obtain that
\rev{
\begin{align*}
    \|\xprodinfty\|_{\w,1}-\Q
    &\leq \frac{\tilde{\g}(\Q,\beta_{\min}) - \tilde{\g}(\Q,\beta_{1})}{\partial_1\tilde{\g}(\Q,\beta_{1})}\\
    &=\begin{cases}
        \frac{\Q(\log(\Q)-1-\log(\beta_{\min})) - \Q(\log(\Q)-1-\log(\beta_{1}))}{\log(\Q) - \log(\beta_{1})}&\text{if }\L=2,\\[6pt]
        \frac{(\Q - \frac{\L}{2}\Q^{\frac{2}{\L}}\beta_{\min}^{1-\frac{2}{\L}}) - (\Q - \frac{\L}{2}\Q^{\frac{2}{\L}}\beta_{1}^{1-\frac{2}{\L}})}{1- \Q^{\frac{2}{\L}-1}\beta_{1}^{1-\frac{2}{\L}}}&\text{if }\L>2
    \end{cases}\\
    &=\Q \cdot 
    \begin{cases}
        \frac{\log(\beta_{1}/\beta_{\min})}{\log(\Q/\beta_{1})} &\text{if }\L=2,\\[6pt]
        \frac{\L(\beta_{1}^{1-\frac{2}{\L}} - \beta_{\min}^{1-\frac{2}{\L}})}{2(\Q^{1-\frac{2}{\L}}- \beta_{1}^{1-\frac{2}{\L}})}&\text{if }\L>2
    \end{cases}.
\end{align*}
}
This completes the proof.
\end{proof}
%


\section{General Case}
\label{sec:Gradient_Flow_General}
This section is dedicated to the proof of Theorem \ref{theorem:L1_equivalence}. Since the proof strategy is very similar to the one of Theorem \ref{theorem:L1_equivalence_positive}, we will not replicate all arguments, but only highlight the key ideas. 
We use the following additional notation in this section.
Let 
\begin{align} \label{def:S}
    S = \{\z\in\mathbb{R}^{\N}:\A\z = \y\}\rev{.}
\end{align}
For $\z\in \mathbb{R}^{\N}$, we denote $\IzPlus=\{n:\ze_n >0\}$ and $\IzMinus=\{n:\ze_n <0\}$ the index sets corresponding to positive and negative entries of $\z$.  We decompose $\z = \z_+ - \z_-$ with
\begin{align*}
    (\ze_+)_n
    = \begin{cases}
    \ze_n &\text{ if } \ze_n > 0, \\ 
    0 &\text{ otherwise},
    \end{cases}
    \qquad \text{ and } \qquad
    (\ze_-)_n
    = \begin{cases}
    -\ze_n &\text{ if } \ze_n < 0, \\ 
    0 &\text{ otherwise.}
    \end{cases}
\end{align*}
We furthermore define
\begin{align} \label{def:Spm}
    \Spm = \{ (\z_+,\,\z_-): \z_+,\,\z_-\geq 0,\, \A(\z_+-\z_-)=\y\}.
\end{align}
as an alternative representation of the solution set $S$.



\subsection{General Bregman divergence}
\label{subsec:Bregman_Divergence_general}


In order to prove that $\u^{\odot\L} - \v^{\odot\L}$ converges to an element in $S$, we will again work with the Bregman Divergence. The main difference is that, instead of $\F$ as defined in \eqref{eq:Breman_positive}, we use the function $\F^{\pm} \colon \mathbb{R}_{+}^{\N} \times \mathbb{R}_{+}^{\N} \to \mathbb{R}$ with
\begin{equation}
    \F^{\pm}(\u,\,\v) = \F(\u) + \F(\v).
\end{equation}
Note that $\F^{\pm}$ is strictly convex because $\F$ is strictly convex and hence the Bregman divergence $D_{\F^\pm} \colon (\mathbb{R}_{+}^{\N} \times \mathbb{R}_{+}^{\N})^2 \to \mathbb{R}$ is well-defined.

\begin{theorem}\label{theorem:Bregman_general}
Let $(\uprod(\t),\,\vprod(\t)) = (\u(\t)^{\odot\L},\,\v(\t)^{\odot\L})$ and
\begin{align}
    \u'(\t) = -\nabla_\u \mathcal{L}^\pm(\u,\,\v),
    \quad 
    \v'(\t) = -\nabla_\v \mathcal{L}^\pm(\u,\,\v)
\end{align}
with $\u(0),\v(0)\geq 0$. Then the limit $(\uprodinfty,\vprodinfty):=\lim_{\t \to \infty} (\uprod(\t),\,\vprod(\t))$ exists and
\begin{equation}\label{eq:optimal_x_general}
    (\uprodinfty,\vprodinfty)
    \in\argmin_{(\z_+,\,\z_-)\in\Spm}
    g_{\uprod(0)}(\z_+) + g_{\vprod(0)}(\z_-)
\end{equation}
where $g$ and $S_\pm$ are defined in \eqref{eq:optimal_x_positive} and \eqref{def:Spm}.
\end{theorem}

\begin{proof}
The proof follows the same steps as the proof of Theorem \ref{theorem:Bregman_positive}, i.e., we start with computing $\partial_\t D_{\F^{\pm}}((\z_+,\,\z_-),\,(\uprod(\t),\,\vprod(\t)))$.
As before, note that due to continuity $\u(\t)$ and $\v(\t)$ remain non-negative for all $\t\geq 0$. Let $\uprod = \u^{\odot\L}$ and $\vprod = \v^{\odot\L}$. Suppose $\A\z = \y$. Decompose $\z$ into $\z_+ - \z_-$ where $\z_+ \ge \0$ and $\z_- \ge \0$ contain the positive resp.\ negative entries of $\z$ in absolute value and are zero elsewhere. We now compute the time derivative of $D_{\F^{\pm}}((\z_+,\,\z_-),(\u(\t),\,\v(\t)))$. By linearity
\begin{align*}
    D_{\F^{\pm}}((\z_+,\,\z_-),\,(\uprod(\t),\,\vprod(\t)))
    &= \F(\z_+) + \F(\z_-) - \F(\uprod(\t)) - \F(\vprod(\t)) \\
    &\quad - \left[ \langle \nabla_{\uprod} \F(\uprod(\t)),\, \z_+ - \uprod(\t) \rangle + \langle \nabla_{\vprod} \F(\vprod(\t)),\, \z_- -\vprod(\t) \rangle \right]\\
    &= D_{\F}(\z_+,\,\uprod(\t)) +  D_{\F}(\z_-,\,\vprod(\t)).
\end{align*}
By a similar calculation as in the proof of Theorem \ref{theorem:Bregman_positive}, we have
\begin{align*}
    \frac{1}{\L}\partial_\t &D_{\F^{\pm}}((\z_+,\,\z_-),\,(\uprod(\t),\,\vprod(\t))) \\
    \qquad &= \langle \A^{\T}( \A(\uprod(\t)-\vprod(\t)) -\y), \z_+-\uprod(\t) \rangle
    - \langle \A^{\T}( \A(\uprod(\t)-\vprod(\t)) -\y), \z_- -\vprod(\t) \rangle\\
    \qquad&= \langle \A(\uprod(\t)-\vprod(\t)) -\y, \A(\z_+-\uprod(\t) - \z_- +\vprod(\t)) \rangle
    = - \| \A(\uprod(\t)-\vprod(\t)) -\y \|_2^2 \\
    \qquad&= -2\mathcal{L}^{\pm}(\uprod(\t)^{\odot\frac{1}{\L}},\vprod(\t)^{\odot\frac{1}{\L}}).
\end{align*}
The same line of argument as in the proof of Theorem \ref{theorem:Bregman_positive} yields $\lim_{\t \to \infty} \mathcal{L}^{\pm}(\uprod(\t)^{\odot\frac{1}{\L}},\,\vprod(\t)^{\odot\frac{1}{\L}})~=~0$, that $\lim_{\t \to \infty} \A(\uprod(\t) - \vprod(\t))~=~\y$, the limit $(\uprodinfty,\vprodinfty):=\lim_{\t \to \infty} (\uprod(\t),\,\vprod(\t))$ exists and both components are non-negative. 
Since the time derivative $\partial_\t D_{\F^{\pm}}$ is identical for all $(\z_+,\,\z_-)\in\Spm$, the difference
\begin{align*}
    \Delta_{\z_+,\,\z_-}
    &= D_{\F^{\pm}}((\z_+,\,\z_-),(\uprod(0),\,\vprod(0))) - D_{\F^{\pm}}((\z_+,\,\z_-),(\uprodinfty,\,\vprodinfty))
\end{align*}
is also identical for all $(\z_+,\,\z_-)\in\Spm$. By non-negativity of $D_{\F^{\pm}}$,
\begin{align*}
    D_{\F^{\pm}}((\z_+,\,\z_-),\,(\uprod(0),\,\vprod(0)))
    \geq \Delta_{\z_+,\,\z_-}
    = \Delta_{\uprodinfty,\,\vprodinfty}
    = D_{\F^{\pm}}((\uprodinfty,\,\vprodinfty),\,(\uprod(0),\,\vprod(0))).
\end{align*}
Expanding the expression we obtain
\begin{align*}
    (\uprodinfty,\vprodinfty)
    &\in\argmin_{(\z_+,\,\z_-)\in\Spm} D_{\F^{\pm}}((\z_+,\,\z_-),\,(\uprod(0),\,\vprod(0)))\\
    &=\argmin_{(\z_+,\,\z_-)\in\Spm} D_{\F}(\z_+,\,\uprod(0)) +  D_{\F}(\z_-,\,\vprod(0))\\
    &=\argmin_{(\z_+,\,\z_-)\in\Spm}
    g_{\uprod(0)}(\z_+) + g_{\vprod(0)}(\z_-)
\end{align*}
by the same calculation as in Theorem \ref{theorem:Bregman_positive}.
\end{proof}

\begin{proof}[Proof of Theorem \ref{theorem:L1_equivalence}]
The proof essentially follows from the proof of Theorem \ref{theorem:L1_equivalence_positive}. Let
\begin{equation*}
    \wprodinfty = \begin{bmatrix}\uprodinfty\\\vprodinfty\end{bmatrix}
    ,\quad \wprod(0) = \begin{bmatrix}\uprod(0)\\\vprod(0)\end{bmatrix}
    ,\quad \tilde{\z} = \begin{bmatrix}\z_+\\\z_-\end{bmatrix}.
\end{equation*}
Let $\w=\wprod(0)^{\odot\frac{2}{\L}-1}$, $\beta_{1} = \|\wprod(0)\|_{\w,1}$, and $\beta_{\min} = \min_{\n\in[\N]}\we_\n\wprode_\n(0)$. By Theorem \ref{theorem:Bregman_general},
\begin{equation*}
    g_{\wprod(0)}(\wprodinfty) \leq g_{\wprod(0)}(\tilde{\z}).
\end{equation*}
Following the same computation as in the proof of Theorem \ref{theorem:L1_equivalence_positive}, we obtain that
\rev{
\begin{align*}
    \|\wprodinfty\|_{\w,1}-\Q
    \leq \frac{\tilde{\g}(\Q,\beta_{\min}) - \tilde{\g}(\Q,\beta_{1})}{\partial_1\tilde{\g}(\Q,\beta_{1})}
    =\Q \cdot 
    \begin{cases}
        \frac{\log(\beta_{1}/\beta_{\min})}{\log(\Q/\beta_{1})} &\text{if }\L=2,\\[6pt]
        \frac{\L(\beta_{1}^{1-\frac{2}{\L}} - \beta_{\min}^{1-\frac{2}{\L}})}{2(\Q^{1-\frac{2}{\L}}- \beta_{1}^{1-\frac{2}{\L}})}&\text{if }\L>2
    \end{cases}
\end{align*}
}
Since $\|\xprodinfty\|_{\w,1} \leq \|\wprodinfty\|_{\w,1}$, the conclusion follows.
\end{proof}
\begin{remark}
Note that $\uprod(0)\neq\vprod(0)$ means that we have different weight for positive and negative components, which gives us a lot of flexibility.
\end{remark}


\section{Numerical Experiment}
\label{sec:Numerical_Experiment}

To evaluate our theoretical results, we conduct several numerical experiments. As outlined below we show recovery of sparse vectors via gradient descent (as approximation of gradient flow) on the overparameterized loss function (or equivalently on the reduced factorized loss function) from the minimal number of Gaussian random measurements as predicted by compressed sensing theory for $\ell_1$-minimization, and thereby
confirm experiments in in previous works \cite{li2021implicit,vaskevicius2019implicit}. Moreover, we conduct experiments in high dimension with random initialization and show that the limit indeed has approximately minimal $\ell_1$ norm as well. 

\subsection{Positive solutions}
\label{subsec:Exp_Positve}


\begin{figure}[t]
    \centering
    \begin{subfigure}[b]{0.45\textwidth}
        \centering
        \includegraphics[width = \textwidth]{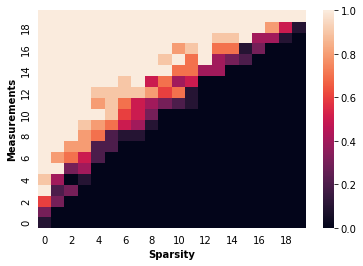}
        \subcaption{$\ell_1$ minimization on $\mathbb{R}_+^{\N}$}
    \end{subfigure}
    \qquad
    \begin{subfigure}[b]{0.45\textwidth}
        \centering
        \includegraphics[width = \textwidth]{Figure/IB_N20_L1_pos.png}
        \subcaption{$\Lover$ minimization via GD ($\L=1$)}
    \end{subfigure}
    \begin{subfigure}[b]{0.45\textwidth}
        \centering
        \includegraphics[width = \textwidth]{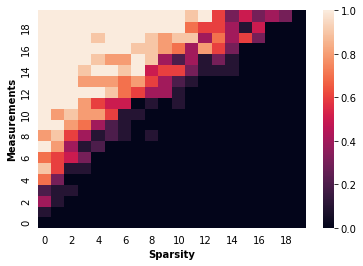}
        \subcaption{$\Lover$ minimization via GD ($\L=2$)}
    \end{subfigure}
    \qquad
    \begin{subfigure}[b]{0.45\textwidth}
        \centering
        \includegraphics[width = \textwidth]{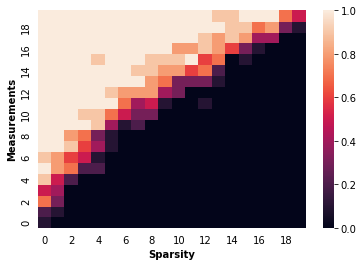}
        \subcaption{$\Lover$ minimization via GD ($\L=3$)}
    \end{subfigure}
    \caption{$\x_*\in\mathbb{R}_+^{\N}$: Comparison of recovery probability between different optimization method.}
    \label{fig:positive}
\end{figure}

In the first experiment, cf.\ Figure \ref{fig:positive}, we consider the simplified setting of Theorem \ref{theorem:L1_equivalence_positive},
where we replace gradient flow with gradient descent in order to have an implementable method. We fix the ambient dimension $\N=20$, vary the number of measurements $\M$ and the sparsity level $s$ of the non-negative ground truth vector $\x_*$ that we wish to recover. The color in the plots represents the probability of successful recovery. Hence, the boundary curves encode the empirically required number of samples in dependence of the sparsity for each algorithm. To ensure fairness, we compare to $\ell_1$-minimization with a positivity constraint on the solution, since we can only recover positive solutions via gradient descent in the setting of Theorem \ref{theorem:L1_equivalence_positive}.

The measurement process $\A\in\mathbb{R}^{\M\times \N}$ and the ground truth $\x_*$ are given by
\begin{align}\label{eq:exp_setting}
    \A = \frac{1}{\sqrt{\M}}{\bf W}, \quad \x_* = \frac{|{\bf b}_S|}{\|{\bf b}_S\|},
\end{align}
where ${\bf W}$ is a standard Gaussian matrix, $S$ is a random subset of $\{1,\dots,\N\}$ of size $s$, which represents the sparsity level, and ${\bf b}_S$ is a standard Gaussian vector supported on $S$. (The index in Figure \ref{fig:positive} and \ref{fig:general} starts from $0$ (instead of $1$) purely due to coding convention.) We take the absolute value of ${\bf b}$ to ensure that the assumption of Theorem \ref{theorem:L1_equivalence_positive} is satisfied. The parameters are set as the following: initialization is set to $\alpha\1$ with $\alpha=10^{-6}$, step size of gradient descent is set to $\eta = 10^{-2}$ and number of iterations equals $T=10^7$. The recovery is considered successful if the resulting error is less than $1\%$, i.e., $\|\xprodinfty - \x_*\|_2\leq 0.01 \|\x_*\|_2$.

We observe that $\ell_1$-minimization (with positivity constraint) and gradient descent on the reduced factorized loss (or equivalently, on the overparameterized loss when initializing identically) behave similarly, whereas gradient descent on the regular quadratic loss barely recovers until $\M=\N$, i.e., when the linear system is fully determined.


\subsection{General solutions}
\label{subsec:Exp_General}


\begin{figure}[t]
    \centering
    \begin{subfigure}[b]{0.45\textwidth}
        \centering
        \includegraphics[width = \textwidth]{Figure/BP_N20_general_2.png}
        \subcaption{$\ell_1$ minimization on $\mathbb{R}^{\N}$}
    \end{subfigure}
    \qquad
    \begin{subfigure}[b]{0.45\textwidth}
        \centering
        \includegraphics[width = \textwidth]{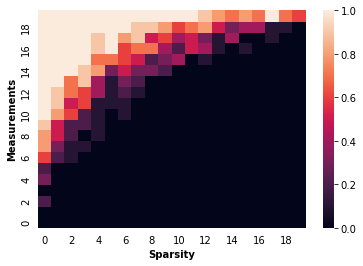}
        \subcaption{$\Lover^{\pm}$ minimization via GD ($\L=2$)}
    \end{subfigure}
    \begin{subfigure}[b]{0.45\textwidth}
        \centering
        \includegraphics[width = \textwidth]{Figure/IB_N20_L3_general_Long.png}
        \subcaption{$\Lover^{\pm}$ minimization via GD ($\L=3$)}
    \end{subfigure}
    \qquad
    \begin{subfigure}[b]{0.45\textwidth}
        \centering
        \includegraphics[width = \textwidth]{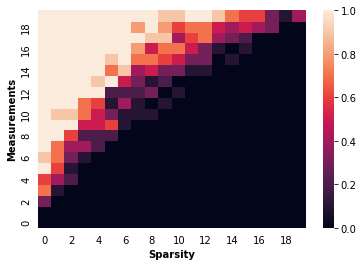}
        \subcaption{$\Lover^{\pm}$ minimization via GD ($\L=4$)}
    \end{subfigure}
    \caption{$\x_*\in\mathbb{R}^{\N}$: Comparison of recovery probability between different optimization method.} 
    \label{fig:general}
\end{figure}

Our second experiment resembles the previous one from Section \ref{subsec:Exp_Positve}, but uses gradient descent on the generalized overparameterized loss function \eqref{eq:L_over_refined} in order to reconstruct general ground truth vectors $\x_*$, which do not need to be non-negative, i.e. $\x_* = \frac{{\bf b}_S}{\|{\bf b}_S\|}$ instead of $\frac{|{\bf b}_S|}{\|{\bf b}_S\|}$. 
We compare with standard $\ell_1$-minimization, i.e., without positivity constraints. As Figure \ref{fig:general} illustrates, the outcome is comparable to the one of the previous experiment, in the sense that GD applied on factorizations performs again similarly to $\ell_1$-minimization.
%

\subsection{Scaling in high dimension}
\label{subsec:Exp_Initialization_Scaling}

%
Recall that $\Q = \min_{\z\in S}\|\z\|_1$. In this experiment we check the influence of the initialization scaling $\alpha$ on the relative optimization error
\begin{equation}\label{eq:exp_setting2}
    \epsilon = \frac{\|\xprodinfty\|_1 - \Q}{\Q}
\end{equation}
in high dimensions. We take $\N=1000$, $\M=150$, and initialization magnitudes $\alpha=10^{-k}$, $k=0,\ldots,6$. The data is generated the same way as in Section \ref{subsec:Exp_General}.

To accelerate the training, we use a line search method to optimize the step size at each iteration. The results in Figure \ref{fig:alpha} show that the accuracy indeed improves as $\alpha$ decreases, which is consistent with Theorem \ref{theorem:L1_equivalence}. The linear behavior in the log-log plot suggests that the relation between error and initialization is indeed polynomial for $\L>2$.

Even though setting $\alpha$ to be small improves the accuracy, such a choice empirically requires more iterations for convergence due to the fact that the origin is a saddle point of the objective function. This yields a trade-off between accuracy and efficiency in the choice of $\alpha$. We observed that the algorithm with least layers, i.e., $\L=2$ converges the fastest but does not give the best results in $\ell_1$-minimization.
\begin{figure}[t]
    \centering
    \begin{subfigure}[b]{0.45\textwidth}
        \centering
        \includegraphics[width = \textwidth]{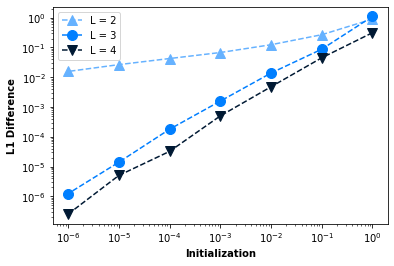}
        \subcaption{Error $\epsilon$ versus initialization in log-log scale for $\L=2,3,4$. For $\L>2$, the curves are roughly linear, which are consistent with Theorem \ref{theorem:L1_equivalence}.}
        \label{fig:alpha}
    \end{subfigure}
    \qquad
    \begin{subfigure}[b]{0.45\textwidth}
        \centering
        \includegraphics[width = \textwidth]{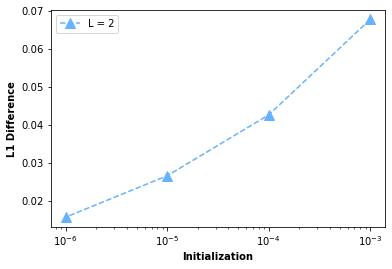}
        \subcaption{Error $\epsilon$ versus initialization in log scale for $\L=2$. Again, the curve is roughly linear, which is consistent with Theorem \ref{theorem:L1_equivalence}.}
        \label{fig:alpha_high_dimension}
    \end{subfigure}
    \caption{Comparison of initialization scaling. Setting: $\N=1000$, $\M=150$. Axes: y-axis = $\epsilon$ defined in \eqref{eq:exp_setting2}, x-axis = $\alpha$ where $\uprod_0=\vprod_0=\alpha\bo$. We can see that for $\L>2$ the error decreases polynomial, while for $\L=2$ the scaling is only logarithmic.}
\end{figure}


\section{Summary and Future Directions}
\label{sec:Summary_Future_Direction}
In the present work, we considered overparametrized square losses for sparse vector recovery from linear measurements. We showed that in these settings vanilla gradient flow exhibits an implicit bias towards solutions that minimize the $\ell_1$-norm among all possible solutions. This lead to near-optimal sampling rates. Several intriguing research questions remain for the future.

First, our theory focuses on gradient flow for minimizing the factorized square loss. However, all conducted experiments use the corresponding gradient descent algorithm instead. A natural next step is thus to extend our results from gradient flow to gradient descent.  

Second, we are also interested in low rank matrix/tensor recovery via overparameterization, which works well empirically but is not yet fully understood. Although sparsity and low-rankness are closely related, one of the significant differences between the Hadamard and the matrix product is that the later is non-commutative in general. In fact, many works \cite{arora2019implicit,chou2020implicit} show that low rankness can be deduced if commutativity is assumed. We believe that extending our proof concepts to the non-commutative matrix case might lead to near-optimal matrix sensing results in this more challenging setting.


\section*{Acknowledgement}
\label{sec:Acknowledgement}
This research was supported through the program "Research in Pairs" by the Mathematisches Forschungsinstitut Oberwolfach (MFO) in 2021. We wish to thank the MFO at this point for their hospitality.
HHC acknowledges funding by the DAAD through the project \emph{Understanding stochastic gradient descent in deep learning} (project no. 57417829).
\rev{We furthermore thank Franziska Schories for helping us to improve the draft by pointing us to some inaccuracies in the proof of Theorem \ref{theorem:L1_equivalence_positive}.}


\section*{Data Availability Statement}
\label{sec:Data Availability Statement}
The data underlying this article will be shared on reasonable request to the corresponding author.


\newpage

\bibliography{references}{}

\begin{thebibliography}{10}

\bibitem{arora2018optimization}
S.~Arora, N.~Cohen, and E.~Hazan.
\newblock On the optimization of deep networks: Implicit acceleration by
  overparameterization.
\newblock In {\em Proceedings of the 35th International Conference on Machine
  Learning, {ICML} 2018}, pages 244--253, 2018.

\bibitem{arora2019implicit}
S.~Arora, N.~Cohen, W.~Hu, and Y.~Luo.
\newblock Implicit regularization in deep matrix factorization.
\newblock In {\em Advances in Neural Information Processing Systems}, pages
  7413--7424, 2019.

\bibitem{azulay2021implicit}
S.~Azulay, E.~Moroshko, M.~S. Nacson, B.~E. Woodworth, N.~Srebro, A.~Globerson,
  and D.~Soudry.
\newblock On the implicit bias of initialization shape: Beyond infinitesimal
  mirror descent.
\newblock In {\em International Conference on Machine Learning}, pages
  468--477, 2021.

\bibitem{Bregman1967}
L.~Bregman.
\newblock The relaxation method of finding the common point of convex sets and
  its application to the solution of problems in convex programming.
\newblock {\em USSR Computational Mathematics and Mathematical Physics},
  7(3):200--217, 1967.

\bibitem{brugiapaglia2021sparse}
S.~Brugiapaglia, S.~Dirksen, H.~C. Jung, and H.~Rauhut.
\newblock Sparse recovery in bounded riesz systems with applications to
  numerical methods for pdes.
\newblock {\em Applied and Computational Harmonic Analysis}, 53:231--269, 2021.

\bibitem{candes2006robust}
E.~J. Cand{\`e}s, J.~Romberg, and T.~Tao.
\newblock Robust uncertainty principles: Exact signal reconstruction from
  highly incomplete frequency information.
\newblock {\em IEEE Transactions on Information Theory}, 52(2):489--509, 2006.

\bibitem{candes2006stable}
E.~J. Cand{\`e}s, J.~K. Romberg, and T.~Tao.
\newblock Stable signal recovery from incomplete and inaccurate measurements.
\newblock {\em Communications on Pure and Applied Mathematics: A Journal Issued
  by the Courant Institute of Mathematical Sciences}, 59(8):1207--1223, 2006.

\bibitem{chou2020implicit}
H.~Chou, C.~Gieshoff, J.~Maly, and H.~Rauhut.
\newblock Gradient descent for deep matrix factorization: Dynamics and implicit
  bias towards low rank.
\newblock {\em arXiv preprint: 2011.13772}, 2020.

\bibitem{donoho2006compressed}
D.~L. Donoho.
\newblock Compressed sensing.
\newblock {\em IEEE Transactions on Information Theory}, 52(4):1289--1306,
  2006.

\bibitem{foucart2013compressed}
S.~Foucart and H.~Rauhut.
\newblock {\em A Mathematical Introduction to Compressive Sensing}.
\newblock Birkhäuser, New York, NY, 2013.

\bibitem{geyer2019implicit}
K.~Geyer, A.~Kyrillidis, and A.~Kalev.
\newblock Low-rank regularization and solution uniqueness in over-parameterized
  matrix sensing.
\newblock In {\em Proceedings of the 23rd International Conference on
  Artificial Intelligence and Statistics}, pages 930--940, 2020.

\bibitem{Bach2019implicit}
G.~Gidel, F.~Bach, and S.~Lacoste-Julien.
\newblock Implicit regularization of discrete gradient dynamics in linear
  neural networks.
\newblock In {\em Advances in Neural Information Processing Systems}, pages
  3202--3211, 2019.

\bibitem{Gissin2019Implicit}
D.~Gissin, S.~Shalev{-}Shwartz, and A.~Daniely.
\newblock The implicit bias of depth: How incremental learning drives
  generalization.
\newblock {\em International Conference on Learning Representations (ICLR).},
  2020.

\bibitem{gunasekar2019implicit}
S.~Gunasekar, J.~D. Lee, D.~Soudry, and N.~Srebro.
\newblock Implicit bias of gradient descent on linear convolutional networks.
\newblock In {\em Advances in Neural Information Processing Systems}, pages
  9461--9471, 2018.

\bibitem{gunasekar2021mirrorless}
S.~Gunasekar, B.~Woodworth, and N.~Srebro.
\newblock Mirrorless mirror descent: A natural derivation of mirror descent.
\newblock In {\em International Conference on Artificial Intelligence and
  Statistics}, pages 2305--2313, 2021.

\bibitem{gunasekar2017implicit}
S.~Gunasekar, B.~E. Woodworth, S.~Bhojanapalli, B.~Neyshabur, and N.~Srebro.
\newblock Implicit regularization in matrix factorization.
\newblock In {\em Advances in Neural Information Processing Systems}, pages
  6151--6159, 2017.

\bibitem{hoff2017lasso}
P.~D. Hoff.
\newblock Lasso, fractional norm and structured sparse estimation using a
  {H}adamard product parametrization.
\newblock {\em Computational Statistics \& Data Analysis}, 115:186--198, 2017.

\bibitem{li2021implicit}
J.~Li, T.~Nguyen, C.~Hegde, and K.~W. Wong.
\newblock Implicit sparse regularization: The impact of depth and early
  stopping.
\newblock In {\em Advances in Neural Information Processing Systems}, 2021.

\bibitem{mendelson2018improved}
S.~Mendelson, H.~Rauhut, and R.~Ward.
\newblock Improved bounds for sparse recovery from subsampled random
  convolutions.
\newblock {\em The Annals of Applied Probability}, 28(6):3491--3527, 2018.

\bibitem{neyshabur2017geometry}
B.~Neyshabur, R.~Tomioka, R.~Salakhutdinov, and N.~Srebro.
\newblock Geometry of optimization and implicit regularization in deep
  learning.
\newblock {\em arXiv preprint: 1705.03071}, 2017.

\bibitem{neyshabur2015}
B.~Neyshabur, R.~Tomioka, and N.~Srebro.
\newblock In search of the real inductive bias: On the role of implicit
  regularization in deep learning.
\newblock In {\em International Conference on Learning Representations}, 2015.

\bibitem{razin2020implicit}
N.~Razin and N.~Cohen.
\newblock Implicit regularization in deep learning may not be explainable by
  norms.
\newblock In {\em Advances in Neural Information Processing Systems}, pages
  21174--21187, 2020.

\bibitem{Razin2021implicitTensor}
N.~Razin, A.~Maman, and N.~Cohen.
\newblock Implicit regularization in tensor factorization.
\newblock {\em arXiv preprint: 2102.09972}, 2021.

\bibitem{Razin2022hierachicalTensor}
N.~Razin, A.~Maman, and N.~Cohen.
\newblock Implicit regularization in hierarchical tensor factorization and deep
  convolutional neural networks.
\newblock {\em CoRR}, abs/2201.11729, 2022.

\bibitem{soudry2018implicit}
D.~Soudry, E.~Hoffer, M.~S. Nacson, S.~Gunasekar, and N.~Srebro.
\newblock The implicit bias of gradient descent on separable data.
\newblock {\em The Journal of Machine Learning Research}, 19(1):2822--2878,
  2018.

\bibitem{stoger2021small}
D.~St{\"o}ger and M.~Soltanolkotabi.
\newblock Small random initialization is akin to spectral learning :
  Optimization and generalization guarantees for overparameterized low-rank
  matrix reconstruction.
\newblock In {\em Advances in Neural Information Processing Systems}, 2021.

\bibitem{vaskevicius2019implicit}
T.~Vaskevicius, V.~Kanade, and P.~Rebeschini.
\newblock Implicit regularization for optimal sparse recovery.
\newblock In {\em Advances in Neural Information Processing Systems}, pages
  2972--2983, 2019.

\bibitem{wang2022large}
Y.~Wang, M.~Chen, T.~Zhao, and M.~Tao.
\newblock Large learning rate tames homogeneity: Convergence and balancing
  effect.
\newblock In {\em International Conference on Learning Representations}, 2022.

\bibitem{woodworth2020kernel}
B.~Woodworth, S.~Gunasekar, J.~D. Lee, E.~Moroshko, P.~Savarese, I.~Golan,
  D.~Soudry, and N.~Srebro.
\newblock Kernel and rich regimes in overparametrized models.
\newblock In {\em Proceedings of Thirty Third Conference on Learning Theory},
  pages 3635--3673, 2020.

\bibitem{wu2020continuous}
F.~Wu and P.~Rebeschini.
\newblock A continuous-time mirror descent approach to sparse phase retrieval.
\newblock {\em Advances in Neural Information Processing Systems},
  33:20192--20203, 2020.

\bibitem{wu2021hadamard}
F.~Wu and P.~Rebeschini.
\newblock Hadamard {W}irtinger flow for sparse phase retrieval.
\newblock In {\em International Conference on Artificial Intelligence and
  Statistics}, pages 982--990. PMLR, 2021.

\bibitem{wu2021implicit}
F.~Wu and P.~Rebeschini.
\newblock Implicit regularization in matrix sensing via mirror descent.
\newblock {\em Advances in Neural Information Processing Systems}, 34, 2021.

\bibitem{zhang17}
C.~Zhang, S.~Bengio, M.~Hardt, B.~Recht, and O.~Vinyals.
\newblock Understanding deep learning requires rethinking generalization.
\newblock In {\em International Conference on Learning Representations}, 2017.

\bibitem{zhao2019implicit}
P.~Zhao, Y.~Yang, and Q.-C. He.
\newblock Implicit regularization via hadamard product over-parametrization in
  high-dimensional linear regression.
\newblock {\em arXiv preprint: 1903.09367}, 2019.

\end{thebibliography}
\bibliographystyle{abbrv}


\newpage
\appendix 

\section{Appendix}
\label{sec:Appendix-Aux}

\subsection{Model reduction}
\label{sec:ModelReduction}



In this section, we formally show that the dynamics of gradient flow on $\Lover$ and $\Lover^\pm$ in \eqref{eq:L_over} and \eqref{eq:L_over_refined} reduce to gradient flow dynamics on the reduced loss functions $\mathcal{L}$ and $\mathcal{L}^\pm$ in \eqref{eq:L} and \eqref{eq:L_refined} if all factors are initialized identically. Although similar statements might already exist in the literature, we provide the full proof for the reader's convenience. We begin with the simpler case of $\Lover$.

\begin{lemma}
\label{lemma:loss_gradient_positive}
    For $\L \geq 2$, let $\xprod:= \bigodot_{\ell\in[\L]}\x^{(\ell)}$ and $\xkc:= \bigodot_{\ell\in[\L]\backslash\{\k\}}\x^{(\ell)}$ for $\k \in [\L]$. Then
    \begin{equation}
    \label{eq:gradient_positive}
        \nabla_{\x^{(\k)}} \Lover\big(\x^{(1)}, \dots, \x^{(\L)}\big)
        = [\A^\T(\A\xprod -\y )] \odot \xkc\;.
    \end{equation}
\end{lemma}
\begin{proof}
    By the chain rule we have, for any $n \in [\N]$, that
    \begin{align*}
        \nabla_{\xe_n^{(\k)}} \Lover\big(\x^{(1)}, \dots, \x^{(\L)}\big)
        &= \frac{1}{2}\sum_{m\in[\M]}\nabla_{\xe_n^{(\k)}} \left(\A\xprod-\y\right)_m^2\\
        &= \sum_{m\in[\M]}(\A\xprod -\y )_m (\A)_{mn} (\xkc)_n
        = [\A^\T(\A\xprod -\y )]_n (\xkc)_n\;.
    \end{align*}
    This completes the proof.
\end{proof}
Using Lemma \ref{lemma:loss_gradient_positive}, 
we now show that the dynamics of the factors $\x^{(\k)}$ can be simplified if all factors are identically initialized.

\begin{lemma}[Identical Initialization] \label{lemma:identical_initialization}
Suppose $\x^{(\k)}(\t)$ follows the negative gradient flow
\begin{equation*}
    \left(\x^{(\k)}\right)'(t) = -\nabla_{\x^{(\k)}} \Lover\big(\x^{(1)}, \dots, \x^{(\L)}\big)\;.
\end{equation*}
If all initialization vectors are identical, i.e., $\x^{(\k)}(0) = \x^{(\k')}(0)$ for all $\k,\k' \in [\L]$, then 
$\x^{(\k)}(\t) = \x^{(\k')}(\t)$ for all $t \geq 0$ and all $\k,\k' \in [\L]$. Moreover, with $\x(t) := \x^{(1)}(t) = \dots = \x^{(\L)}(t)$ and $\mathcal{L}(\x) = \frac{1}{2}\|\A\x^{\odot L} -\y\|_2^2$ the dynamics is given by
\begin{equation*}
    \x'(\t)
    = -\nabla\mathcal{L}(\x)\;.
\end{equation*}
\end{lemma}
\begin{proof}
It suffices to show that if $\x^{(\k)} = \x^{(\k')}$ for all $\k,\k' \in [\L]$, then $\nabla_{\x^{(\k)}}\Lover = \nabla_{\x^{(\k')}}\Lover$ for all $\k,\k' \in [\L]$. By Lemma \ref{lemma:loss_gradient_positive}, the only dependence of $\nabla_{\x^{(\k)}}\Lover$ on $\k$ is through $\xkc$. Since $\x^{(\k)}$ and $\x^{(\k')}$ are identical,
\begin{align*}
    \xkc
    &= \bigodot_{\ell\in[\L]\backslash\{\k\}}\x^{(\ell)}
    = \bigodot_{\ell\in[\L]\backslash\{\k'\}}\x^{(\ell)}
    = \x_{(\k')^c}
\end{align*}
and hence $\nabla_{\x^{(\k)}}\Lover = \nabla_{\x^{(\k')}}\Lover$. Since by assumption all $\x^{(\k)}$ are identical, we can replace them with $\x$. Hence $\xprod = \x^{\odot L}$ and $\xkc = \x^{\odot (L-1)}$. Plugging this into \eqref{eq:gradient_positive}, we get that $\x'(\t) = -\left[\A^{\T}( \A\x^{\odot \L} -\y)\right] \odot \x^{\odot (\L-1)}$, which is exactly $\nabla\mathcal{L}(\x)$.
\end{proof}

At first sight, the reduction of the number of parameters in Lemma \ref{lemma:identical_initialization} may seem counter-intuitive to the idea of overparameterization. However, as opposed to the standard loss function 
$\Lquad(\x) = \frac{1}{2}\|\A\x -\y\|_2^2$, we arrive at an alternative loss function $\mathcal{L}(\x) = \frac{1}{2}\|\A\x^{\odot \L} -\y\|_2^2$ induced by the overparameterization and having a different optimization landscape than $\Lquad$.




Along similar lines it is straight-forward to check that the gradient flow of the general overparameterized loss function $\Lover^\pm$ in \eqref{eq:L_over_refined} is equivalent to the flow of the reduced factorized loss $\mathcal{L}^\pm$ in \eqref{eq:L_refined} if all $\u^{(\k)}$ and $\v^{(\k)}$ are initialized identically. 
Note that the partial gradients of $\mathcal{L}^\pm$ are given by 
    \begin{align*}
        \nabla_\u \mathcal{L}^\pm(\u,\v) &= \L \left[\A^{\T}( \A(\u^{\odot \L} - \v^{\odot L}) -\y)\right] \odot \u^{\odot L-1} \\
        \text{and} \qquad
        \nabla_\v \mathcal{L}^\pm(\u,\v) &= - \L \left[\A^{\T}( \A(\u^{\odot \L} - \v^{\odot L}) -\y)\right] \odot \v^{\odot L-1}.
    \end{align*}

\subsection{Noise robust compressed sensing}
\label{subsec:Noise robust compressed sensing}

For completeness we provide a proof of Theorem~\ref{thm:noisy-cs}. 
\begin{proof}[Proof of Theorem~\ref{thm:noisy-cs}]
The $\ell_2$-robust null space property of order $c \s_*$ with respect to $\ell_2$ with constants $\rho \in (0,1)$ and $\tau$ as in \eqref{robust-nsp} together with the $\ell_1$-quotient property implies the so-called simultaneous $(\ell_2,\ell_1)$-quotient property by \cite[Lemma 11.15]{foucart2013compressed}, i.e., for any $\e \in \R^\M$ there exists $\u \in \R^\M$ such that $\A \e = \u$ and both
\begin{equation}\label{simul-quotient-prop}
\|\u\|_2 \leq d' \|\e\|_2 \quad \mbox{ and } \quad \|\u\|_1 \leq d \sqrt{\s_*} \|\e\|_2,
\end{equation}
where the constant $d'$ only depends on $d, c, \rho, \tau$.  Let $\Delta : \R^M \to \R^N$ the reconstruction map represented by gradient flow, i.e., $\Delta(\y')$ is the limit of gradient flow (initialized with parameter $\u_0,\v_0$) for the functional $\Lover^\pm$, the latter depending on $\y'$. 
We know from Theorem~\ref{theorem:L1_equivalence} that $\A(\Delta(\y)) = \y$ for any $\y$ (the limit $\xprodinfty$ is contained in $S$).
By \cite[Theorem 4.25]{foucart2013compressed} the $\ell_2$-robust null space property of $\A$ implies
that
\begin{equation}\label{robust-nsp-consequence}
\|\v - \w\|_2 \leq \frac{C}{\sqrt{s}}\left(\|\v\|_1 -\|\w\|_1 + 2 \sigma_{s}(\w)_1\right) + D \|\A(\v-\w)\|_2 \quad \mbox{ for all } \v,\w \in \R^\N.
\end{equation}
For the noise vector $\e$ of the theorem, let $\u$ be the vector such that $\A \u = \e$ and \eqref{simul-quotient-prop} holds. Note that $\y = \A \x_* + \e = \A(\x_+ + \u)$. Further let, $\widehat{\z}$ be a minimizer of $\min_{\z : \A \z = \A(\x_* + \u)} \|\z\|_1$. By Theorem~\ref{theorem:L1_equivalence} we have
\begin{equation}\label{bound:main:thm}
\|\w\|_1 - \|\widehat{\z}\|_1 \leq \epsilon\|\widehat{\z}\|_1.
\end{equation}
Applying inequality \eqref{robust-nsp-consequence} for $\v = \x_*+\u$ and $\w = \Delta(\A(\x_*+\u))$ below (noting that $\A\v = \A \w$ as argued above) yields
\begin{align*}
\|\x_* - \Delta(\y)\|_2 
& = \|\x_* - \Delta(\A(\x_* + \u)\|_2 \leq \|\x_* + \u - \Delta(\A(\x_+ + \u))\|_2 + \|\u\|_2 \\
& \leq \frac{C}{\sqrt{s}} \left(\|\w\|_1 - \|\x_* + \u\|_1 + 2 \sigma_s(\x_* + \u)_1\right) + \|\u\|_2\\
& \leq \frac{C}{\sqrt{s}} \left(\|\w\|_1 - \|\widehat{\z}\|_1 + \|\widehat{\z}\|_1 -  \|\x_* + \u\|_1 + 2 \sigma_s(\x_*)_1 + \|\u\|_1\right) + \|\u\|_2.
\end{align*}
Since $\widehat{\z}$ minimizes the $\ell_1$-norm among all vectors satisfying $\A\z = \A(\x_* + \u)$ we have $\|\x_* + \u\|_1 \geq \|\widehat{\z}\|_1$. Using \eqref{simul-quotient-prop} and \eqref{bound:main:thm}
and that $\s = c \s_*$ gives
\begin{align*}
    \|\x_* - \Delta(\y)\|_2
    &\leq \frac{C}{\sqrt{\s}}\left(\epsilon\|\widehat{\z}\|_1 + 2 \sigma_s(\x_*)_1\right) + \frac{Cd\sqrt{\s_*}}{\sqrt{\s}} \|\e\|_2 + d' \|\e\|_2\\
    &= \frac{C}{\sqrt{\s}}\left(\epsilon\|\widehat{\z}\|_1+ 2 \sigma_s(\x_*)_1\right) + C' \|\e\|_2
\end{align*}
with $C' = Cd\sqrt{\s_*}/\sqrt{\s} + d'$. This completes the proof of Theorem~\ref{thm:noisy-cs}.
\end{proof}


\end{document}